%% file: main.tex
\documentclass{article}

\input{preamble}

\title{Bounds on the Posterior-to-Prior Ratios for Inclusion Belief under Bounded Differential Privacy}

\author[1,$\dagger$]{Jan Reiter Sørensen}
\author[1]{Heidi Søgaard Christensen}
\author[1,2]{Rasmus Rask Kragh Jørgensen}
\author[1]{Martin Bøgsted}

\affil[1]{Center for Clinical Data Science, Aalborg University, Aalborg, Denmark}
\affil[2]{Department of Hematology, Aalborg University Hospital, Aalborg, Denmark}
\affil[$\dagger$]{Corresponding author: \href{mailto:janrs@dcm.aau.dk}{\textcolor{blue}{janrs@dcm.aau.dk}}}

\date{\today}

\begin{document}

    \maketitle

    \centerline{\textit{\textbf{Abstract}}}
        \input{abstract}

    \section{Introduction}
        \input{introduction}

    \section{Notation and Preliminaries}\label{sec:preliminaries}
        \input{notationAndPreliminaries}

    \section{Probabilistic Differential Privacy}\label{sec:relationship}
        \input{relationshipADPAndPDP}

    \section{Posterior-to-Prior-Ratio of Inclusion Belief}\label{sec:inclusionbelief}
        \input{inclusionBelief}

    \section{The Risk of Failing the Ratio Bounds under the Gaussian Mechanism}\label{sec:failurerate}
        \input{failingRisk}

    \section{The Empirical Failure Rate}\label{sec:empirical}
        \input{empirical}

    \section{Conclusion}\label{sec:conclusion}
        \input{conclusion}

    \bibliography{litteratur}

\end{document}

%% file: preamble.tex
\usepackage{graphicx} 
\usepackage{tikz}
\usepackage[a4paper, margin=1in]{geometry}
\usepackage{amssymb}
\usepackage{amsmath}
\usepackage{mathrsfs}
\usepackage{mathtools}
\usepackage[parfill]{parskip}
\usepackage{bbm}
\usepackage[toc,page]{appendix}
\usepackage{bm}
\usepackage{xcolor}
\usepackage{subcaption}
\definecolor{indigo}{HTML}{211A52}

\usepackage{algorithm}
\usepackage{algpseudocode}
\usepackage{booktabs}
\usepackage{authblk}

\usepackage{amsthm}
\usepackage[hidelinks]{hyperref}
\newtheorem{theorem}{Theorem}
\newtheorem{definition}{Definition}
\newtheorem{lemma}{Lemma}

\newtheorem{corollary}{Corollary}

\DeclarePairedDelimiter\norm{\lVert}{\rVert}

\newcommand\restr[2]{{
  \left.\kern-\nulldelimiterspace 
  #1 
  \littletaller 
  \right|_{#2} 
  }}

\newcommand{\littletaller}{\mathchoice{\vphantom{\big|}}{}{}{}}

\allowdisplaybreaks



\usepackage[round]{natbib}

%% file: abstract.tex
\textit{Differential privacy has become the standard for generating privacy-protected data releases. However, differential privacy does not translate intuitively to disclosure risk. In particular, it remains unclear how much an adversary's belief about an individual's inclusion in a dataset can change after observing a protected release. To address this question, we derive upper and lower bounds on the posterior-to-prior ratios of inclusion beliefs under bounded probabilistic and approximate differential privacy. By assuming a worst-case adversary with all-but-one auxiliary information, i.e., knowledge of all except for one of the participants in a dataset, we obtain bounds that apply to any adversary. Because these bounds may fail with non-zero probability, we study the corresponding failure probability for the Gaussian mechanism. We derive a theoretical upper limit on this probability and compare it with Monte Carlo estimates across a wide range of parameter settings. The observed failure rate is several orders of magnitude smaller than its theoretical upper limit, indicating that the latter is highly conservative. These findings suggest that the inferential privacy guarantees provided by differentially private mechanisms may be substantially stronger in practice than what is implied by the theoretical upper limit.}

%% file: introduction.tex
Differential privacy, first introduced by \citet{dwork2006calibrating}, has become a widely accepted standard for privacy-preserving data sharing. It has been adopted across a broad range of applications, including statistical databases \citep{christ2022differential}, web-browser telemetry collection \citep{erlingsson2014rappor}, and the training of large language models \citep{mckenna2025scaling}. By limiting the information that can be learned about an individual, differential privacy provides strong theoretical guarantees regarding disclosure risk and inference attacks. However, the practical interpretation of these guarantees remains elusive. In practice, users must select a privacy budget, yet determining an appropriate value remains challenging because privacy budgets lack a direct and intuitive interpretation in terms of privacy risk \citep{dwork2019differential}. Consequently, researchers have often turned to Bayesian adversarial frameworks to assess privacy risk \citep{wagner2018technical}. Within these frameworks, privacy risk is characterised through an adversary's posterior beliefs after observing a differentially private release. \citet{mcclure2012differential} argue that posterior probabilities of disclosure risk and inference attacks alone are insufficient for assessing privacy risk, as they do not account for the adversary's prior knowledge. Rather, the change in posterior relative to prior beliefs provides a more informative metric, as it reflects the information gained from the release of differentially private data. Therefore, characterising posterior-to-prior changes is fundamental to understanding the privacy risk under differential privacy.

Such changes can be quantified using posterior-to-prior differences and ratios \citep{jarmin2023depth}. Using these measures, \citet{kazan2024prior} derive deterministic upper and lower bounds on attribute and membership inference under unbounded pure differential privacy. \citet{kazan2025bayesian} extends these results to unbounded approximate differential privacy, obtaining probabilistic bounds that may fail with non-zero probability. Complementing this work, \citet{bailie2024general} establish general Bayesian inferential limits under pure differential privacy, yielding direct bounds on posterior-to-prior ratios for arbitrary inferences drawn from the observed data.

The bounds derived by \citet{kazan2024prior} for attribute and membership inference cover two important classes of disclosure risks. Nevertheless, adversarial inferences are not limited to these settings and also include risks such as re-identification \citep{jarmin2023depth}, data reconstruction \citep{guerra2026understanding}, and other forms of inference \citep{pilgram2025consensus}. The general inferential bound of \citet{bailie2024general} encompasses this broader spectrum of disclosure risks by applying to arbitrary Bayesian inference on differentially private releases. However, the methods of \citet{kazan2024prior} and \citet{bailie2024general} are limited to pure differential privacy, a notion known to exhibit a steep privacy-utility trade-off, where even marginal increases in privacy budget can substantially degrade statistical accuracy of a differentially private release \citep{canonne2020discrete}. Consequently, pure differential privacy is often replaced with more relaxed definitions, such as approximate differential privacy. \citet{kazan2025bayesian} subsequently extended the attribute and membership inference bounds to the unbounded approximate differential privacy setting. Despite the widespread adoption of bounded differential privacy in practice \citep{erlingsson2014rappor, christ2022differential, hayes2023bounding}, no corresponding posterior-to-prior bounds have been established for bounded approximate differential privacy.

In this paper, we partially fill this gap by deriving bounds on the posterior-to-prior ratios for inclusion beliefs under bounded approximate differential privacy. Similar to \citet{kazan2025bayesian}, our bounds are probabilistic and may therefore fail with non-zero probability. For the Gaussian mechanism, we characterise the regions of the output space where the bounds are not guaranteed to hold. This characterisation enables closed-form expressions for the probability mass of these regions, providing an upper limit on the probability that the proposed guarantees fail. Finally, we use Monte Carlo simulations to evaluate the tightness of this upper limit and find that it is highly conservative in practice.

The paper is organised as follows. \autoref{sec:preliminaries} introduces the mathematical framework and relevant privacy definitions. \autoref{sec:relationship} covers the relation between probabilistic and approximate differential privacy and derives the shape and mass of the failure regions for the Gaussian mechanism. \autoref{sec:inclusionbelief} establishes the posterior-to-prior ratio bounds under bounded probabilistic and bounded approximate differential privacy, while \autoref{sec:failurerate} derives a closed-form expression for an upper limit on their failure probability for the Gaussian mechanism. \autoref{sec:empirical} presents a Monte Carlo study evaluating the tightness of this limit. Finally, \autoref{sec:conclusion} discusses the findings, addresses limitations, and outlines directions for future research.

%% file: notationAndPreliminaries.tex
Our notation and terminology largely follow \citet{bailie2024general}. The following definition of a \textit{data universe} is adapted from their framework. 

\begin{definition}\textbf{(Data universe)}\label{def:datauniverse}\newline
    Let $\mathcal{X}$ be a set and $G\coloneqq(\mathcal{X},E)$ an undirected graph with unit-length edges. A data universe is an extended metric space $(\mathcal{X},d)$, where the metric $d:\mathcal{X}\times\mathcal{X}\to[0,\infty]$ is given by
    \begin{align*}
        d(x,x')\coloneqq\begin{cases}
            0,\quad &\text{if }x=x',\\
            k,\quad &\text{if $x\neq x'$ are connected in $G$ and $k$ is the length of the} \\
            & \text{shortest path between them,}\\
            \infty,\quad &\text{if $x,x'$ are not connected in $G$,}
        \end{cases}
    \end{align*}
    for $x,x'\in\mathcal{X}$.
\end{definition}

Intuitively, $\mathcal{X}$ represents the set of all datasets that may be drawn from a given population. A dataset $x\in\mathcal{X}$ is typically a set of records sampled from a common domain $\mathcal{D}$, that is, $x=\{r_{i}\}_{i=1}^{n}$ with $r_{i}\in\mathcal{D}$ for all $i=1,2,\ldots,n$. As an example, $\mathcal{D} = \mathbb{R}^2_{\geq0}$ may represent height and weight measurements of the population. Note for notational simplicity the graph $G$ is omitted from the data universe notation, as it is typically redundant when $d$ is given. The metric $d$ determines the notion of neighbouring datasets, with $x$ and $x'$ being neighbours whenever $d(x,x')=1$. Common examples used in differential privacy include the Hamming distance and the symmetric difference, which induce different neighbour relations.

The following definition introduces \textit{data-release mechanisms}, which are randomised mappings that transform datasets into into publicly released outputs, such as statistics, synthetic datasets, or other privacy-protected data products.

\begin{definition}\textbf{(Data-release mechanism)}\label{def:datareleasemechanism}\newline
    Let $(\mathcal{X},d)$ be a data universe and $(\mathcal{T},\mathscr{F})$ a standard Borel space. The map $M:\mathcal{X}\times[0,1]\to\mathcal{T}$ is called a data-release mechanism if $M(x,\cdot)$ is $\mathscr{B}[0,1]$-$\mathscr{F}$-measurable for all $x\in\mathcal{X}$. Additionally, we define a family of probability measures $\left\{P_{x}\right\}_{x\in\mathcal{X}}$ on $(\mathcal{T},\mathscr{F})$, where
    \begin{align}\label{eq:probmeasure}
        P_{x}(S)\coloneqq\lambda\left(\left\{u\in[0,1]\mid M(x,u)\in S\right\}\right)
    \end{align}
    for $x\in\mathcal{X}$ and $S\in\mathscr{F}$, and where $\lambda$ denotes the Lebesgue measure.
\end{definition}

The notion of data-release mechanisms encompasses both randomised and deterministic mechanisms. Examples of deterministic outputs include summary statistics such as sample means, regression coefficients, or $p$-values computed from a dataset. Randomised mechanisms incorporate an additional source of randomness, so that repeated applications to the same dataset may produce different outputs. The Gaussian mechanism, introduced below, provides an example.

\textit{Differential privacy} is a property that quantifies the level of privacy a data-release mechanism guarantees. The following definition is adapted from \citet{dwork2006our}.

\begin{definition}\textbf{(Approximate differential privacy)}\label{def:adp}\newline
    Let $(\mathcal{X},d)$ be a data universe and $M:\mathcal{X}\times[0,1]\to\mathcal{T}$ a data-release mechanism. For $\varepsilon,\delta\in\mathbb{R}_{\geq0}$, the data-release mechanism is said to satisfy $(\varepsilon,\delta)$-approximate differential privacy if
    \begin{align}\label{eq:approximateDP}
        P_{x}(S)\leq e^{\varepsilon}P_{x'}(S)+\delta
    \end{align}
    for $x,x'\in\mathcal{X}$ with $d(x,x')=1$ and $S\in\mathscr{F}$. When $\delta=0$, $M$ is said to satisfy pure $\varepsilon$-differential privacy.
\end{definition}

Two commonly used notions are \emph{unbounded} (or \emph{add/remove-one}) differential privacy, where neighbouring datasets differ by the presence of a single individual, and \emph{bounded} (or \emph{replace-one}) differential privacy, where neighbouring datasets have the same size and differ in the record of a single individual. The choice of neighbour relation affects the sensitivity of statistical queries and consequently the amount of noise required to achieve a given privacy level. To accommodate both settings, the definition above is stated with respect to a generic neighbour relation $d$.

Under pure differential privacy, symmetry of the neighbour relation implies that
\begin{align*}
    e^{-\varepsilon}P_{x'}(S)\leq P_x(S)\leq e^{\varepsilon}P_{x'}(S),
\end{align*}
meaning that the output distributions of the data-release mechanism differ by at most a factor of $e^{\varepsilon}$. The $\delta$ parameter of approximate differential privacy relaxes these bounds by allowing that
\begin{align*}
    e^{\varepsilon}P_{x'}(S)<P_x(S)\leq e^{\varepsilon}P_{x'}(S)+\delta,
\end{align*}
as long as $\delta>0$. Note that $\delta\geq1$ would render the privacy promise of approximate differential privacy vacuous.

\subsubsection*{Example: The Gaussian mechanism}
    To illustrate the preceding definitions, we present a simple example based on the Gaussian mechanism.
    
    Consider datasets consisting of $m$ records, where each record is a point in $\mathbb{R}^{2}$ representing, for instance, height and weight measurements. We then define
    \begin{align}\label{eq:domainGaussianMechanism}
        \mathcal{X}=\restr{2^{\mathbb{R}^{2}}}{m} 
    \end{align}
    as the collection of all datasets of size $m$. We equip $\mathcal{X}$ with the Hamming distance,
    \begin{align}\label{eq:Hammeren}
        d(x,x')\coloneqq\sum_{i\in x}\mathbbm{1}_{i\notin x'},
    \end{align}
    for  $x,x'\in\mathcal{X}$. Since all datasets in $\mathcal{X}$ have cardinality $m$, every pair of datasets is connected. With the above, we have now specified a data universe $(\mathcal{X},d)$.
    
    Suppose we observe a dataset $x\in\mathcal{X}$ and wish to publish the output $f(x)$ of a deterministic query $f:\mathcal{X}\to A$, where $A$ is a bounded subset of $\mathbb{R}^{2}$. For instance, $f(x)$ could be the bivariate sample mean of height and weight measurements. Releasing $f(x)$ may reveal sensitive information, and we are therefore interested in publishing it through a privacy-protecting data-release mechanism. The Gaussian mechanism proposed by \cite{dwork2014algorithmic} privatises the output of a query by adding normally distributed noise, that is,
    \begin{align}\label{eq:gaussianMechanism}
        M(x,u)\coloneqq f(x)+\sigma\Phi^{-1}(u),
    \end{align}
    where the entries of $\Phi^{-1}:[0,1]\to\mathbb{R}^{2}$ are the inverse standard normal distribution function. According to \cite{dwork2014algorithmic}, this mechanism satisfies $(\varepsilon,\delta)$-approximate differential privacy if
    \begin{align}\label{eq:ClassicalGaussian}
        \sigma>\frac{\Delta_{2}f\sqrt{2\log(1.25/\delta)}}{\varepsilon},
    \end{align}
    where $\varepsilon,\delta\in]0,1[$ and $\Delta_{2}f\in\mathbb{R}_{\geq0}$ is the $\ell_{2}$-sensitivity of $f$ defined by
    \begin{align*}
        \Delta_{2}f\coloneqq\sup_{\substack{x,x'\in\mathcal{X}\\d(x,x')=1}}\norm{f(x)-f(x')}_{2}.
    \end{align*}
    Since the Hamming distance was used to define neighbouring datasets, the Gaussian mechanism satisfies bounded $(\varepsilon,\delta)$-approximate differential privacy. If we had used the symmetric difference instead, the mechanism would satisfy unbounded $(\varepsilon,\delta)$-approximate differential privacy. Note, for this example we considered the classical Gaussian mechanism, however, other versions of the Gaussian mechanism have been proposed by, for instance, \cite{balle2018improving} and \cite{ji2024less}. 

%% file: relationshipADPAndPDP.tex
To derive bounds on the posterior-to-prior ratio under bounded approximate differential privacy, we we adapt the approach of \citet{kazan2025bayesian}, who considered the unbounded setting. Specifically, we first derive bounds under bounded \textit{probabilistic differential privacy} and then transfer these results to bounded approximate differential privacy using the implication from approximate differential privacy to probabilistic differential privacy. In this section, we restate probabilistic differential privacy and its relation to approximate differential privacy within our mathematical framework.

Probabilistic differential privacy is a variant of differential privacy that relaxes the strictness of pure $\varepsilon$-differential privacy \citep{machanavajjhala2008privacy}. Instead of adding a $\delta$, as approximate differential privacy does in \eqref{eq:approximateDP}, probabilistic differential privacy is defined as pure $\varepsilon$-differential privacy that breaks with probability at most $\delta$. The following definition is adapted from \citet{meiser2018approximate}.

\begin{definition}\textbf{(Probabilistic differential privacy)}\label{def:pdp}\newline
    Let $(\mathcal{X},d)$ be a data universe and $M:\mathcal{X}\times[0,1]\to\mathcal{T}$ a data-release mechanism. For $\varepsilon,\delta\in\mathbb{R}_{\geq0}$, the data-release mechanism $M$ is said to satisfy $(\varepsilon,\delta)$-probabilistic differential privacy if, for $x,x'\in\mathcal{X}$ with $d(x,x')=1$, there exists a set $S_{x,x'}^{\delta}\in\mathscr{F}$ with $P_{x}(S_{x,x'}^{\delta})\leq\delta$, such that 
    \begin{align}\label{eq:probabilisticDP}
            e^{-\varepsilon}P_{x'}(S\setminus S_{x,x'}^{\delta})\leq P_{x}(S\setminus S_{x,x'}^{\delta})\leq e^{\varepsilon}P_{x'}(S\setminus S_{x,x'}^{\delta})
    \end{align}
    for $S\in\mathscr{F}$. All such sets $S^{\delta}_{x,x'}$ are called bad regions.
\end{definition}

Similarly to approximate differential privacy, a value of $\delta\geq 1$ renders \autoref{def:pdp} vacuous as the privacy condition is then allowed to fail with probability one. The following theorem, adapted from \citet{zhao2019reviewing}, establishes that probabilistic differential privacy directly implies approximate differential privacy. 

\begin{theorem}\textbf{}\label{thm:PDPtoADP}\newline
    For any $\varepsilon,\delta\in\mathbb{R}_{\geq 0}$, if a data-release mechanism satisfies $(\varepsilon,\delta)$-probabilistic differential privacy it also satisfies $(\varepsilon,\delta)$-approximate differential privacy.
\end{theorem}

\begin{proof}\textbf{}\newline
    Let $(\mathcal{X},d)$ be a data universe and $M:\mathcal{X}\times[0,1]\to\mathcal{T}$ a data-release mechanism satisfying $(\varepsilon,\delta)$-probabilistic differential privacy. Further, let $x,x'\in\mathcal{X}$ with $d(x,x')=1$ and a set $S\in\mathscr{F}$ be given. By assumption, there exists a set $S_{x,x'}^{\delta}\in\mathscr{F}$, such that $P_{x}(S_{x,x'}^{\delta})\leq\delta$ and $\eqref{eq:probabilisticDP}$ hold. Consequently, we obtain
    \begin{align*}
        P_{x}(S)&=P_{x}(S\setminus S_{x,x'}^{\delta})+P_{x}(S\cap S_{x,x'}^{\delta})\\
        &\leq e^{\varepsilon}P_{x'}(S\setminus S_{x,x'}^{\delta})+P_{x}(S_{x,x'}^{\delta})\\
        &\leq e^{\varepsilon}P_{x'}(S)+\delta,
    \end{align*}
    establishing \eqref{eq:approximateDP}.
\end{proof}

The converse relation is not immediate. In particular, it requires proving the existence of a measurable set $S_{x,x'}^{\delta}\in\mathscr{F}$ satisfying \eqref{eq:probabilisticDP} and determining an upper bound on $P_{x}(S_{x,x'}^{\delta})$. We have therefore formulated the following lemma that states the existence of such a set and serves as a key ingredient in the proof for the converse of \autoref{thm:PDPtoADP}.

\begin{lemma}\textbf{}\label{lem:existence}\newline
    Let $(\mathcal{X},d)$ be a data universe,  $M:\mathcal{X}\times[0,1]\to\mathcal{T}$ a data-release mechanism, and $\varepsilon\in\mathbb{R}_{\geq0}$. For all $x,x'\in\mathcal{X}$, there exists a set $S_{x,x'}\in\mathscr{F}$, such that
    \begin{align}\label{eq:smallestSx}
        e^{-\varepsilon}P_{x'}(S\setminus S_{x,x'})\leq P_{x}(S\setminus S_{x,x'})&\leq e^{\varepsilon}P_{x'}(S\setminus S_{x,x'})
    \end{align}
    for $S\in\mathscr{F}$, and for any set $S_{x,x'}'\in\mathscr{F}$ satisfying \eqref{eq:smallestSx} in the place of $S_{x,x'}$, we have $P_x(S_{x,x'})\leq P_x(S_{x,x'}')$.
\end{lemma}

\begin{proof}\textbf{}\newline
    Let $x,x'\in\mathcal{X}$ be given and let $\mu$ be any $\sigma$-finite measure such that $\mu\gg P_{x}$ and $\mu\gg P_{x'}$. Such a measure always exists, for instance by taking $\mu= P_{x}+P_{x'}$. Regardless of how $\mu$ is constructed, the corresponding densities
    \begin{align*}
        p\coloneqq\frac{dP_{x}}{d\mu}\text{ and } q\coloneqq\frac{dP_{x'}}{d\mu}
    \end{align*}
    are well defined. Define $S_{x,x'}\coloneqq A_{x,x'}\cup B_{x,x'}$, where the disjoint subsets $A_{x,x'}$ and $B_{x,x'}$ are given by
    \begin{equation}\label{eq:AxBx}
        \begin{aligned}
            A_{x,x'}&\coloneqq\left\{v\in\mathcal{T}\mid e^{-\varepsilon}q(v)>p(v)\right\},\\
            B_{x,x'}&\coloneqq\left\{v\in\mathcal{T}\mid p(v)>e^{\varepsilon}q(v)\right\}.
        \end{aligned}
    \end{equation}
    Since $p$ and $q$ are measurable, both $A_{x,x'}$ and $B_{x,x'}$ belong to $\mathscr{F}$. Hence, so does $S_{x,x'}$. By construction $ e^{-\varepsilon}q(v)\le p(v)\le e^\varepsilon q(v)$ for all $v\notin S_{x,x'}$. Therefore, for $S\in\mathscr{F}$, 
    \begin{align*}
        P_{x}(S\setminus S_{x,x'})&=\int_{S\setminus S_{x,x'}}p(v)d\mu(v)\leq e^{\varepsilon}\int_{S\setminus S_{x,x'}}q(v)d\mu(v)=e^{\varepsilon}P_{x'}(S\setminus S_{x,x'}), \\
        P_{x}(S\setminus S_{x,x'})&=\int_{S\setminus S_{x,x'}}p(v)d\mu(v)\geq e^{-\varepsilon}\int_{S\setminus S_{x,x'}}q(v)d\mu(v)=e^{-\varepsilon}P_{x'}(S\setminus S_{x,x'}).
    \end{align*}
    This proves the existence of a set $S_{x,x'}$ satisfying \eqref{eq:smallestSx}. 
    It remains to show that $S_{x,x'}$ minimises $P_x(\cdot)$ among all such sets. To this end, let $S_{x,x'}'\in\mathscr{F}$ be any other set satisfying \eqref{eq:smallestSx}, and define $D\coloneqq S_{x,x'}\setminus S_{x,x'}'$. Then
    \begin{align*}
        P_{x}(D)=P_{x}(A_{x,x'}\setminus S_{x,x'}')+P_{x}(B_{x,x'}\setminus S_{x,x'}').
    \end{align*}
    It holds that $A_{x,x'}\setminus S_{x,x'}'$ and $B_{x,x'}\setminus S_{x,x'}'$ are $\mu$-zero sets, since
    \begin{align*}
        P_{x}(A_{x,x'}\setminus S_{x,x'}')\geq e^{-\varepsilon}P_{x'}(A_{x,x'}\setminus S_{x,x'}')&\quad \Leftrightarrow \quad \int_{A_{x,x'}\setminus S_{x,x'}'}\left(p(v)-e^{-\varepsilon}q(v)\right)d\mu(v)\geq0,\\
        P_x(B_{x,x'}\setminus S_{x,x'}')\leq e^{\varepsilon}P_{x'}(B_{x,x'}\setminus S_{x,x'}')&\quad \Leftrightarrow \quad \int_{B_{x,x'}\setminus S_{x,x'}'}\left(p(v)-e^{\varepsilon}q(v)\right)d\mu(v)\leq0,
    \end{align*}
    and the integrands $p(v)-e^{-\varepsilon}q(v)$ and $p(v)-e^{\varepsilon}q(v)$ are, respectively, strictly negative and positive. Consequently, $D$ is a $\mu$-zero set and therefore also a $P_{x}$-zero set. Hence, 
    \begin{align*}
        P_{x}(S_{x,x'})=P_{x}(S_{x,x'}\cap S_{x,x'}')+P_{x}(S_{x,x'}\setminus S_{x,x'}')=P_{x}(S_{x,x'}\cap S_{x,x'}')\leq P_{x}(S_{x,x'}').
    \end{align*}
    This proves that $S_{x,x'}$ has minimal  $P_{x}$-measure among all sets satisfying \eqref{eq:smallestSx}.
\end{proof}

In essence, \autoref{lem:existence} identifies a minimal bad region for each neighbouring pair of datasets.
Using this result, we can prove the following theorem, which states the converse relation of \autoref{thm:PDPtoADP}. A proof for this result is given by \citet{zhao2019reviewing}, however, we provide an alternative proof.

\begin{theorem}\textbf{}\label{thm:ADPtoPDP}\newline
    Let $\varepsilon\in\mathbb{R}_{\geq0}$ and $\delta\in[0,1)$. If a mechanism satisfies $(\varepsilon,\delta)$-approximate differential privacy, it also satisfies $(\varepsilon',\delta')$-probabilistic differential privacy for any $\varepsilon'>\varepsilon$ and
    \begin{align*}
        \delta'=\frac{\delta(1+e^{-\varepsilon'})}{1-e^{\varepsilon-\varepsilon'}}.
    \end{align*}
\end{theorem}

\begin{proof}\textbf{}\newline
    Let $x,x'\in\mathcal{X}$ with $d(x,x')=1$ be given. We need to show that there exists a set $S_{x,x'}^{\delta'}\in\mathscr{F}$, such that \eqref{eq:probabilisticDP} holds and
    \begin{align}\label{eq:deltaprime}
        P_{x}(S_{x,x'}^{\delta'})\leq\delta'\coloneqq\frac{\delta(1+e^{-\varepsilon'})}{1-e^{\varepsilon-\varepsilon'}}.
    \end{align}
    The existence of a set satisfying \eqref{eq:probabilisticDP} for $\varepsilon'$ is concluded by \autoref{lem:existence}, so it remains to show \eqref{eq:deltaprime}. We will assume that $S_{x,x'}^{\delta'}$ is the smallest set satisfying \eqref{eq:probabilisticDP} for a given $\varepsilon'>\varepsilon$. As in the proof of \autoref{lem:existence}, let $\mu$ be a $\sigma$-finite measure for which $\mu\gg P_{x}$ and $\mu\gg P_{x'}$, and define the densities
    \begin{align*}
        p\coloneqq\frac{dP_{x}}{d\mu}, \quad q\coloneqq\frac{dP_{x'}}{d\mu}.
    \end{align*}
    Let $A_{x,x'}$ and $B_{x,x'}$ be given as in \eqref{eq:AxBx}. Then, under $(\varepsilon, \delta)$-approximate differential privacy, we have
    \begin{align*}
        P_{x}(B_{x,x'})&\leq e^{\varepsilon}P_{x'}(B_{x,x'})+\delta\\
        &= e^{\varepsilon}\int_{B_{x,x'}}q(v)d\mu(v)+\delta\\
        &\leq e^{\varepsilon-\varepsilon'}\int_{B_{x,x'}}p(v)d\mu(v)+\delta\\
        &=e^{\varepsilon-\varepsilon'}P_{x}(B_{x,x'})+\delta,
    \end{align*}
    which implies that
    \begin{align}\label{eq:upperboundBx}
        P_{x}(B_{x,x'})\leq\frac{\delta}{1-e^{\varepsilon-\varepsilon'}}.
    \end{align}
    Consider now the set $S_{x',x}^{\delta'}\in\mathscr{F}$ obtained by exchanging $x$ and $x'$ in \autoref{def:pdp} and \autoref{lem:existence}. Proceeding analogously to the construction above, we decompose $S_{x',x}^{\delta'}$ into $A_{x',x},B_{x',x}\in\mathscr{F}$. By construction, $A_{x,x'}=B_{x',x}$ and $B_{x,x'}=A_{x',x}$. Since \eqref{eq:upperboundBx} holds for arbitrary $x\in\mathcal{X}$, we must also have 
    \begin{align*}
        P_{x'}(B_{x',x})\leq\frac{\delta}{1-e^{\varepsilon-\varepsilon'}}.
    \end{align*}
    This gives the following upper bound:
    \begin{align*}
        P_{x}(A_{x,x'})&=\int_{A_{x,x'}}p(v)d\mu(v)\\
        &\leq e^{-\varepsilon'}\int_{A_{x,x'}}q(v)d\mu(v)\\
        &=e^{-\varepsilon'}P_{x'}(A_{x,x'})\\
        &=e^{-\varepsilon'}P_{x'}(B_{x',x})\\
        &\leq e^{-\varepsilon'}\frac{\delta}{1-e^{\varepsilon-\varepsilon'}}.
    \end{align*}
    Consequently, 
    \begin{align*}
        P_{x}(S_{x,x'}^{\delta'})&=P_{x}(A_{x,x'})+P_{x}(B_{x,x'})\\
        &\leq e^{-\varepsilon'}\frac{\delta}{1-e^{\varepsilon-\varepsilon'}}+\frac{\delta}{1-e^{\varepsilon-\varepsilon'}}\\
        &=\frac{\delta(1+e^{-\varepsilon'})}{1-e^{\varepsilon-\varepsilon'}},
    \end{align*}
    which was to be shown.
\end{proof}

\autoref{thm:ADPtoPDP} shows how the $(\varepsilon,\delta)$-parameters of approximate differential privacy translates into corresponding parameters $\varepsilon'$ and $\delta'$ for probabilistic differential privacy. Specifically, $\varepsilon'$ may be chosen arbitrarily larger than $\varepsilon$, after which $\delta'$ can be calculated in closed form. The following corollary, adapted from \citet{kazan2025bayesian}, provides the inverse characterisation: $\delta'$ may be chosen arbitrarily within a suitable interval, and the corresponding value of $\varepsilon'$ is then given in closed form.

\begin{corollary}\textbf{}\label{cor:adptopdp}\newline
    A data-release mechanism satisfying $(\varepsilon,\delta)$-approximate differential privacy for $\varepsilon\in\mathbb{R}_{\geq0}$ and $\delta\in[0,1)$ also satisfies $(\varepsilon',\delta')$-probabilistic differential privacy for any $\delta'\in(\delta,1]$ and
    \begin{align*}
        \varepsilon'=\log\left(\frac{\delta'e^{\varepsilon}+\delta}{\delta'-\delta}\right).
    \end{align*}
\end{corollary}

\begin{proof}\textbf{}\label{cor:alternativeEpsilonDelta}\newline
    Let $\delta'\in(\delta,1]$ be given. Solving 
    \begin{align*}
    \delta' = \frac{\delta(1+e^{-\varepsilon'})} {1-e^{\varepsilon-\varepsilon'}}
    \end{align*} 
    for $\varepsilon'$ yields 
    \begin{align*}
        \varepsilon' = \log\left( \frac{\delta'e^{\varepsilon}+\delta} {\delta'-\delta} \right),
    \end{align*}
    which was to be shown.
\end{proof}

\subsubsection*{The Geometry of the bad regions under the Gaussian mechanism}
    In this section we investigate the shape and size of the bad regions $S^{\delta}_{x,x'}$ for the Gaussian mechanism. As stated in \autoref{sec:preliminaries}, the Gaussian mechanism satisfies approximate differential privacy. By \autoref{thm:ADPtoPDP} it also satisfies probabilistic differential privacy. 

    Let $(\mathcal{X},d)$ be a data universe, and consider again the Gaussian mechanism given by \eqref{eq:gaussianMechanism}, but where $f:\mathcal{X}\to A$ and $A$ is a bounded subset of $\mathbb{R}^{q}$. In \autoref{sec:preliminaries} we specified the data universe by \eqref{eq:domainGaussianMechanism} and \eqref{eq:Hammeren}, however, the Gaussian mechanism satisfies approximate differential privacy regardless of the data universe. According to \autoref{def:pdp} there exists a set $S^{\delta}_{x,x'}\in\mathscr{B}(\mathbb{R}^{q})$, such that
    \begin{align*}
        P_{x}(S^{\delta}_{x,x'})\leq\delta
    \end{align*}
    and \eqref{eq:probabilisticDP} holds for any $S\in\mathscr{B}(\mathbb{R}^{q})$. Under the Gaussian mechanism, 
    \begin{align*}
        P_{x}(S)=\frac{1}{\sigma^q}\int_{S}\phi\left(\frac{v - f(x)}{\sigma}\right)dv
        \quad \text{and} \quad 
        P_{x'}(S)=\frac{1}{\sigma^q}\int_{S}\phi\left(\frac{v - f(x')}{\sigma}\right)dv
    \end{align*}
    for all $S\in\mathscr{B}(\mathbb{R}^{q})$, where $\phi$ denotes density of the $q$-dimensional standard normal distribution. Hence, the measures $P_{x}$ and $P_{x'}$ admit the densities $p$ and $q$ given by
    \begin{align*}
        p(v)\coloneqq \frac{1}{\sigma^q}\phi\left(\frac{v - f(x)}{\sigma}\right)\quad \text{and} \quad 
        q(v)\coloneqq \frac{1}{\sigma^q} \phi\left(\frac{v - f(x')}{\sigma}\right).
    \end{align*}
    Since the support of $\phi$ is $\mathbb{R}^{q}$, we have
    \begin{equation}\label{eq:Sdeltaxxprime}
        \begin{aligned}
            S^{\delta}_{x,x'}&=\left\{v\in\mathbb{R}^{q} \, \mid \, p(v)<e^{-\varepsilon}q(v) \, \vee \, e^{\varepsilon}q(v)<p(v)\right\}\\[0.75ex]
            &=\left\{v\in\mathbb{R}^{q}\mid \left|\log\left(\frac{p(v)}{q(v)}\right)\right|>\varepsilon\right\},
        \end{aligned}
    \end{equation}
    where
    \begin{align*}
        \log\left(\frac{p(v)}{q(v)}\right)&=\log\left(\frac{\exp\left(-\frac{1}{2\sigma^{2}}\norm{v-f(x)}^{2}_{2}\right)}{\exp\left(-\frac{1}{2\sigma^{2}}\norm{v-f(x')}^{2}_{2}\right)}\right)\\[0.75ex]
        &=\frac{1}{2\sigma^{2}}\left(\norm{v-f(x')}^{2}_{2}-\norm{v-f(x)}^{2}_{2}\right)
    \end{align*}
    with $\norm{\cdot}_{2}$ denoting the Euclidean norm. When $f(x)=f(x')$, we get $S^{\delta}_{x,x'}=\emptyset$. For $f(x)\neq f(x')$, define $\Delta f\coloneqq f(x)-f(x')$ and $m\coloneqq\frac{f(x)+f(x')}{2}$. Thus, $f(x)=m+\frac{\Delta f}{2}$ and $f(x')=m-\frac{\Delta f}{2}$, and we can now rewrite
    \begin{align*}
        \norm{v-f(x')}^{2}_{2}-\norm{v-f(x)}^{2}_{2}&=\norm{v-m+\frac{\Delta f}{2}}^{2}_{2}-\norm{v-m-\frac{\Delta f}{2}}^{2}_{2}\\[0.75ex]
        &=\left\langle v-m+\frac{\Delta f}{2},v-m+\frac{\Delta f}{2}\right\rangle-\left\langle v-m-\frac{\Delta f}{2},v-m-\frac{\Delta f}{2}\right\rangle\\[0.75ex]
        &=2\left\langle v-m,\Delta f\right\rangle.
    \end{align*}
    Combining the above yields
    \begin{align*}
       \log\left(\frac{p(v)}{q(v)}\right)=\frac{\left\langle v-m,\Delta f\right\rangle}{\sigma^{2}}=\frac{\left\langle v-m,\tilde{u}\right\rangle\Delta}{\sigma^{2}},
    \end{align*}
    where $\Delta\coloneqq\norm{\Delta f}_{2}$ and $\tilde{u}\coloneqq\Delta f/\Delta$. Substituting this into \eqref{eq:Sdeltaxxprime}, we obtain 
    \begin{align}\label{eq:badRegionShape}
        S^{\delta}_{x,x'}&=\left\{v\in\mathbb{R}^{q} \,\mid \, \left|\frac{\left\langle v-m,\tilde{u}\right\rangle\Delta}{\sigma^{2}}\right|>\varepsilon\right\}\nonumber\\[0.75ex]
        &=\left\{v\in\mathbb{R}^{q} \,\mid  \,\Big|\left\langle v,\tilde{u}\right\rangle-\left\langle m,\tilde{u}\right\rangle\Big|>\frac{\varepsilon\sigma^{2}}{\Delta}\right\}.
    \end{align}
    Thus, $S^{\delta}_{x,x'}$ is the union of the two open half-spaces defined by the inequalities $\left\langle v,\tilde{u}\right\rangle>\frac{\varepsilon\sigma^{2}}{\Delta}+\left\langle m,\tilde{u}\right\rangle$ and $\left\langle v,\tilde{u}\right\rangle<-\frac{\varepsilon\sigma^{2}}{\Delta}+\left\langle m,\tilde{u}\right\rangle$. Let $V(u)\coloneqq M(x,u)$. Using \eqref{eq:badRegionShape}, we can rewrite
    \begin{align}\label{eq:P_x_lambda}
        P_{x}(S^{\delta}_{x,x'})&=P_{x}\left(\left\{v\in\mathbb{R}^{q}\mid \, \Big|\left\langle v,\tilde{u}\right\rangle-\left\langle m,\tilde{u}\right\rangle\Big|>\frac{\varepsilon\sigma^{2}}{\Delta}\right\}\right)\nonumber \\ 
        &=\lambda\left(\left\{u\in[0,1] \,\mid \, \Big|\left\langle V(u),\tilde{u}\right\rangle-\left\langle m,\tilde{u}\right\rangle\Big|>\frac{\varepsilon\sigma^{2}}{\Delta}\right\}\right)\nonumber\\ 
        &=\lambda\left(\Big|\left\langle V,\tilde{u}\right\rangle-\left\langle m,\tilde{u}\right\rangle\Big|>\frac{\varepsilon\sigma^{2}}{\Delta}\right)\nonumber\\ 
        &=\lambda\left(\left\langle V,\tilde{u}\right\rangle-\left\langle m,\tilde{u}\right\rangle>\frac{\varepsilon\sigma^{2}}{\Delta}\right)+\lambda\left(\left\langle V,\tilde{u}\right\rangle-\left\langle m,\tilde{u}\right\rangle<-\frac{\varepsilon\sigma^{2}}{\Delta}\right). 
    \end{align}
    Since $V(u)=M(x,u)$, we have by the construction of the Gaussian mechanism in \eqref{eq:gaussianMechanism},
    \begin{align*}
        \left\langle V(u),\tilde{u}\right\rangle&=\left\langle f(x)+\sigma\Phi^{-1}(u),\tilde{u}\right\rangle\\[0.75ex]
        &=\left\langle f(x),\tilde{u}\right\rangle+\sigma\left\langle\Phi^{-1}(u),\tilde{u}\right\rangle\\[0.75ex]
        &=\left\langle m+\frac{\Delta f}{2},\tilde{u}\right\rangle+\sigma\left\langle\Phi^{-1}(u),\tilde{u}\right\rangle\\[0.75ex]
        &=\left\langle m+\frac{\Delta\cdot\tilde{u}}{2},\tilde{u}\right\rangle+\sigma\left\langle\Phi^{-1}(u),\tilde{u}\right\rangle\\[0.75ex]
        &=\left\langle m,\tilde{u}\right\rangle+\frac{\Delta}{2}+\sigma\left\langle\Phi^{-1}(u),\tilde{u}\right\rangle.
    \end{align*}
    Then, defining $N(u)\coloneqq\left\langle\Phi^{-1}(u),\tilde{u}\right\rangle$, which is standard normally distributed, \eqref{eq:P_x_lambda} becomes
    \begin{align}\label{eq:exactBadSize}
        P_{x}(S^{\delta}_{x,x'})&=\lambda\left(\left\langle m,\tilde{u}\right\rangle+\frac{\Delta}{2}+\sigma N-\left\langle m,\tilde{u}\right\rangle>\frac{\varepsilon\sigma^{2}}{\Delta}\right)+\lambda\left(\left\langle m,\tilde{u}\right\rangle+\frac{\Delta}{2}+\sigma N-\left\langle m,\tilde{u}\right\rangle<-\frac{\varepsilon\sigma^{2}}{\Delta}\right)\nonumber\\[0.75ex]
        &=\lambda\left(\frac{\Delta}{2}+\sigma N>\frac{\varepsilon\sigma^{2}}{\Delta}\right)+\lambda\left(\frac{\Delta}{2}+\sigma N<-\frac{\varepsilon\sigma^{2}}{\Delta}\right)\nonumber\\[0.75ex]
        &=\lambda\left(-N<-\frac{\varepsilon\sigma}{\Delta}+\frac{\Delta}{2\sigma}\right)+\lambda\left(N<-\frac{\varepsilon\sigma}{\Delta}-\frac{\Delta}{2\sigma}\right)\nonumber\\[0.75ex]
        &=\Phi\left(-\frac{\varepsilon\sigma}{\Delta}+\frac{\Delta}{2\sigma}\right)+\Phi\left(-\frac{\varepsilon\sigma}{\Delta}-\frac{\Delta}{2\sigma}\right).
    \end{align}
    This can be used to exactly calculate the size of the bad region for any choice of $x,x'\in\mathcal{X}$ with $d(x,x')=1$.

%% file: inclusionBelief.tex
In this section, we quantify the inferential capabilities of a Bayesian adversary under differential privacy. We focus on inclusion belief, defined as the adversary's subjective probability that a specific individual is contained in the underlying dataset. The adversary begins with a prior belief about inclusion, which is updated to a posterior belief upon observing information about the dataset. A natural decision rule is then to infer inclusion based on this posterior probability. However, it has been argued that posterior beliefs alone do not fully capture privacy risk \citep{hotz2024key}. Recent works such as \cite{kazan2024prior} and \cite{bailie2024general} have opted for posterior-to-prior ratios and differences instead. Individually, these measures do not convey the exact risk of disclosure but how much the belief of an adversary changes due to the release of dataset information. For a data holder deciding whether to release information about a dataset, the implied posterior-to-prior change is more informative than the posterior belief itself; if the change is negligible, then so is the additional privacy risk induced by the release, and vice versa. Posterior beliefs are conditioned on the outputs of a data-release mechanism, and since these outputs may follow a continuous distribution, we set up a probability theoretic framework for defining probability measures conditioned on realisations of continuous random variables.

\subsubsection*{Defining the posterior-to-prior ratio}

We assume for the rest of this paper that an observed dataset consists of a sample drawn from a finite population. Specifically, let $P\coloneqq\left\{1,2,\ldots,n\right\}$ denote a population of $n\in\mathbb{N}$ individuals, and let $m\in\{1,\ldots,n\}$ denote the fixed sample size. Define 
$$\mathcal{X}\coloneqq\restr{2^{P}}{m},$$
the collection of all possible datasets of size $m$ that can be observed from the population. For the rest of this paper, we will equip $\mathcal{X}$ with the Hamming distance $d$ defined in \eqref{eq:Hammeren}, such that $(\mathcal{X},d)$ forms a data universe. Hence, any data-release mechanism satisfying \autoref{def:adp} satisfies bounded approximate differential privacy from now on. Let $M$ be a data-release mechanism as in \autoref{def:datareleasemechanism}, and assume that the family of probability measures $\left\{P_{x}\right\}_{x\in\mathcal{X}}$ on $(\mathcal{T},\mathscr{F})$ given by $\eqref{eq:probmeasure}$ is dominated by a $\sigma$-finite measure $\nu$. In this setup, a dataset $x\in\mathcal{X}$ is represented as a subset of the population indices $P$. In practice, however, each individual is associated with a record of attributes taking values in some domain $\mathcal D$, and the data-release mechanism $M$ operates on these attribute values rather than on the indices themselves. As in \cite{kazan2025bayesian}, we assume that each individual in $P$ can be uniquely identified by its attribute record, that is, there exists a bijection between $P$ and $\mathcal{D}$ that is embedded in $M$. Consequently, we may represent a dataset by the set of individual identifiers, while $M$ internally maps these identifiers to their corresponding attribute values when producing a release. This representation is applicable regardless of whether we consider cancer registries, blood panel databases, or other collections of independent static records, provided that no two records are exactly identical.

From the perspective of a Bayesian adversary, the observed dataset is random. Hence, on some probability space $(\Omega,\mathscr{A},\mathbb{P})$, we define a random variable $X:\Omega\to\mathcal{X}$ for which a realisation represents an observed dataset. Since $P_{x}\ll\nu$ for all $x\in\mathcal{X}$, each measure $P_{x}$ admits a density $p_{x}$ with respect to $\nu$. Using these densities, we define the intensity measure as
\begin{align}\label{eq:intensitymeasure}
    P_{M,X}(t,I)\coloneqq\int_{I}p_{X(\omega)}(t)d\mathbb{P}(\omega),
\end{align}
where $t\in\mathcal{T}$ and $I\in\mathscr{A}$. Formally, \eqref{eq:intensitymeasure} requires the map $(x,t)\mapsto p_{x}(t)$ to be measurable. This follows from assuming that $M$ is $(\mathscr{A}\otimes\mathscr{B}[0,1])$-$\mathscr{F}$ measurable, in addition to the assumptions of \autoref{def:datareleasemechanism}.
Define the conditional density of the data-release mechanism by
\begin{align*}
    P_{M\mid X}(t\mid I)\coloneqq\frac{P_{M,X}(t,I)}{\mathbb{P}(I)},
\end{align*}
where $t\in\mathcal{T}$, $I\in\mathscr{A}$, and $\mathbb{P}(I)>0$. Finally, we define the conditional probability measure of $X$ given an approximate differentially private mechanism output $t\in\mathcal{T}$ and a set of auxiliary information $\mathcal{A}\in\mathscr{A}$ by
\begin{align}\label{eq:conditionalmeasure}
    P_{X|M}(I\mid t,\mathcal{A})\coloneqq\frac{P_{M,X}(t, I\cap\mathcal{A})}{P_{M,X}(t,\mathcal{A})},
\end{align}
where $P_{M,X}(t,\mathcal{A})>0$. This conditional probability measure reflects the probability that the observed dataset $x$ is sampled from $X(I)$ given a mechanism output $t\in\mathcal{T}$ and some auxiliary information about the observed dataset, represented by $\mathcal{A}$. Assuming that the adversary has auxiliary information is, as noted by \citet{guerra2026understanding}, crucial as bounding the capabilities of such an adversary will also bound the capabilities of adversaries with less information. For inclusion inference, we consider the event
\begin{align}\label{eq:inclusionInference}
    I_{i}\coloneqq\left\{\omega\in\Omega\mid i\in X(\omega)\right\}
\end{align}
for $i\in P$, which contains all realisations of the data-generating process for which individual $i$ is in the observed dataset.
For the auxiliary information, we adopt an all-but-one framework and assume that the adversary knows $m-1$ individuals in the observed dataset, i.e.
\begin{align}\label{eq:auxiliaryEvent}
    \mathcal{A}\coloneqq\left\{\omega\in\Omega\mid x^{-}\subset X(\omega)\right\}
\end{align}
for some fixed $x^{-}\subset P\setminus\left\{i\right\}$ of length $m-1$. 
Our goal is to assess upper and lower bounds to the following posterior-to-prior ratio under bounded approximate differential privacy
\begin{align}\label{eq:postpriorratio}
    \mathscr{R}(t)\coloneqq\frac{P_{X\mid M}(I_{i}\mid t,\mathcal{A})}{\mathbb{P}(I_{i}\mid\mathcal{A})}.
\end{align}

\subsubsection*{Posterior-to-prior ratio bounds under bounded differential privacy}
Before deriving the bounds, we need the following two results. The first states that probabilistic differential privacy has a pointwise equivalent version.

\begin{theorem}\textbf{}\label{thm:pointwisepDP}\newline
    Let $M$ be a data-release mechanism, and assume that its family of probability measures $\left\{P_{x}\right\}_{x\in\mathcal{X}}$ defined by \eqref{eq:probmeasure} is dominated by a $\sigma$-finite measure $\nu$. Let $p_{x}$ and $p_{x'}$ denote the densities of $P_{x}$ and $P_{x'}$, respectively, with respect to $\nu$. Then $M$ satisfies $(\varepsilon,\delta)$-probabilistic differential privacy for $\varepsilon,\delta\in\mathbb{R}_{\geq0}$ if and only if, for any $x,x'\in\mathcal{X}$ with $d(x,x')=1$ there exists a set $S^{\delta}_{x,x'}\in\mathscr{F}$ with $P_{x}(S^{\delta}_{x,x'})\leq\delta$, where
    \begin{align}\label{eq:pointwisepDP}
        e^{-\varepsilon}p_{x'}\leq p_{x}\leq e^{\varepsilon}p_{x'}
    \end{align}
    $\nu$-almost everywhere on $\mathcal{T}\setminus S^{\delta}_{x,x'}$.
\end{theorem}

\begin{proof}
    We first prove that probabilistic differential privacy implies \eqref{eq:pointwisepDP}. Assume that $M$ satisfies $(\varepsilon,\delta)$-probabilistic differential privacy, and let $x,x'\in\mathcal{X}$ be given with $d(x,x')=1$, where $(\mathcal{X},d)$ is our data universe. By the definition of probabilistic differential privacy, there exists a set $S^{\delta}_{x,x'}\in\mathscr{F}$ with $P_{x}(S^{\delta}_{x,x'})\leq\delta$ such that \eqref{eq:probabilisticDP} holds for any $S\in\mathscr{F}$. Then
    \begin{align*}
        \int_{S\setminus S^{\delta}_{x,x'}}p_{x}(t)d\nu(t)=P_{x}(S\setminus S^{\delta}_{x,x'})\leq e^{\varepsilon}P_{x'}(S\setminus S^{\delta}_{x,x'})=\int_{S\setminus S^{\delta}_{x,x'}}e^{\varepsilon}p_{x'}(t)d\nu(t),
    \end{align*}
    where $p_{x}$ and $p_{x'}$ are densities of $P_{x}$ and $P_{x'}$, respectively, with respect to $\nu$. Since this inequality holds for arbitrary $S\in\mathscr{F}$, $p_{x}(t)\leq e^{\varepsilon}p_{x'}(t)$ for $\nu$-almost every $t\in \mathcal{T}\setminus S^{\delta}_{x,x'}$. The converse inequality can be derived analogously. This proves the first implication.

    To prove the reverse implication, let $x,x'\in\mathcal{X}$ with $d(x,x')=1$, and assume that there exists some set $S^{\delta}_{x,x'}\in\mathscr{F}$ with $P_{x}(S^{\delta}_{x,x'})\leq\delta$ such that \eqref{eq:pointwisepDP} holds for $\nu$-almost every $t\in \mathcal{T}\setminus S^{\delta}_{x,x'}$. Then, for $S\in\mathscr{F}$, 
    \begin{align*}
        P_{x}(S\setminus S^{\delta}_{x,x'})=\int_{S\setminus S^{\delta}_{x,x'}}p_{x}(t)d\nu(t)\leq\int_{S\setminus S^{\delta}_{x,x'}}e^{\varepsilon}p_{x'}(t)d\nu(t)=e^{\varepsilon}P_{x'}(S\setminus S^{\delta}_{x,x'}).
    \end{align*}
    It can similarly be shown that $P_{x}(S\setminus S^{\delta}_{x,x'})\geq e^{-\varepsilon}P_{x'}(S\setminus S^{\delta}_{x,x'})$, which proves $(\varepsilon,\delta)$-probabilistic differential privacy of $M$.
\end{proof}

\autoref{thm:pointwisepDP} shows that probabilistic differential privacy can be characterised pointwise whenever the family of probability measures defined by the data-release mechanism is dominated by a common $\sigma$-finite measure. Assuming that a mechanism satisfies the conditions of \autoref{thm:pointwisepDP}, then there exists by \autoref{lem:existence} a tight smallest set $S^{\delta}_{x,x'}\in\mathscr{F}$ constructed through \eqref{eq:AxBx} for which the inequality \eqref{eq:pointwisepDP} fails everywhere on $S^{\delta}_{x,x'}$ and holds everywhere on its complement. Hence, we can understand the bad region $S^{\delta}_{x,x'}$ as the exact region in which the data-release mechanism violates pure $\varepsilon$-differential privacy. Releasing an output in $\mathcal{T}\setminus S^{\delta}_{x,x'}$ is therefore a necessary and sufficient condition for \eqref{eq:pointwisepDP} to hold.

Using \autoref{thm:pointwisepDP}, we establish the following lemma, which forms the basis for deriving upper and lower bounds on $\mathscr{R}$ under probabilistic differential privacy.

\begin{lemma}\textbf{}\label{lem:probBounds}\newline
   Let $I_{i}$ and $\mathcal{A}$ denote the inclusion and auxiliary-information events defined in \eqref{eq:inclusionInference} and  \eqref{eq:auxiliaryEvent}, respectively, and assume that $\mathbb{P}(I_{i}^{c}\cap\mathcal{A})>0$. Let $M$ be a data-release mechanism satisfying $(\varepsilon,\delta)$-probabilistic differential privacy for $\varepsilon,\delta\in\mathbb{R}_{\geq0}$, and assume that its family of probability measures $\left\{P_{x}\right\}_{x\in\mathcal{X}}$ given by \eqref{eq:probmeasure} is dominated by a $\sigma$-finite measure $\nu$. If $x\in X(I_{i}\cap\mathcal{A})$, then
    \begin{align}\label{eq:lemineq}
        e^{-\varepsilon}\leq\frac{P_{M\mid X}(t\mid I_{i}^{c}\cap\mathcal{A})}{p_{x}(t)}\leq e^{\varepsilon}
    \end{align}
    for $\nu$-almost every $t\in\mathcal{T}\setminus S^{\delta}_{x}$, where
    \begin{align*}
        S^{\delta}_{x}\coloneqq\bigcup_{x'\in X(I_{i}^{c}\cap\mathcal{A})}S^{\delta}_{x,x'}.
    \end{align*}
\end{lemma}

\begin{proof}
    Rewriting the numerator, we obtain that
    \begin{align*}
        P_{M\mid X}(t\mid I_{i}^{c}\cap\mathcal{A})&=\frac{P_{M,X}(t,I^{c}_{i}\cap\mathcal{A})}{\mathbb{P}(I^{c}_{i}\cap\mathcal{A})}\\
        &=\frac{1}{\mathbb{P}(I_{i}^{c}\cap\mathcal{A})}\int_{I^{c}_{i}\cap\mathcal{A}}p_{X(\omega)}(t)d\mathbb{P}(\omega)\\
        &\stackrel{\dagger}{\leq}\frac{1}{\mathbb{P}(I_{i}^{c}\cap\mathcal{A})}\int_{I^{c}_{i}\cap\mathcal{A}}e^{\varepsilon}p_{x}(t)d\mathbb{P}(\omega)\quad\nu\text{-almost every }t\in\mathcal{T}\setminus S^{\delta}_{x}\\
        &=e^{\varepsilon}p_{x}(t),
    \end{align*}
    where $\dagger$ follows from the fact that every realisation $X(\omega)$ is a neighbour of $x$ for $\omega\in I^{c}_{i}\cap\mathcal{A}$, and thus \autoref{thm:pointwisepDP} applies. This establishes the right-hand inequality in \eqref{eq:lemineq}; the left-hand inequality follows analogously.
\end{proof}

In \autoref{lem:probBounds}, the condition $t\in\mathcal{T}\setminus S^{\delta}_{x}$ is sufficient, but not necessary, for \eqref{eq:lemineq} to hold. To see this, consider the integrals on the left- and right-hand sides of the inequality marked $\dagger$ in the proof of \autoref{lem:probBounds}. Since $X(I^{c}_{i}\cap\mathcal{A})$ is a finite set, we may write
\begin{align*}
   X(I^{c}_{i}\cap\mathcal{A}) =  \{x_{1},x_{2},\ldots,x_{q}\},
\end{align*}
for some $x_j \in \mathcal{X}$, $j = 1, \dots, q$. The integrals can then be expressed as
\begin{align*}
    \int_{I^{c}_{i}\cap\mathcal{A}}p_{X(\omega)}(t)d\mathbb{P}(\omega)&=\sum_{j=1}^{q}p_{x_{j}}(t)(\mathbb{P}\circ X^{-1})(\{x_{j}\}),\\
     \int_{I^{c}_{i}\cap\mathcal{A}}e^{\varepsilon}p_{x}(t)d\mathbb{P}(\omega)&=\sum_{j=1}^{q}e^{\varepsilon}p_{x}(t)(\mathbb{P}\circ X^{-1})(\{x_{j}\}).
\end{align*}
Assuming a uniform prior on $X(I_i^c\cap\mathcal A)$, we have \begin{align*} (\mathbb P\circ X^{-1})(\{x_j\}) = \frac{\mathbb P(I_i^c\cap\mathcal A)}{q}, \qquad j=1,\ldots,q. \end{align*} It follows that \begin{align*} P_{M\mid X}(t\mid I_i^c\cap\mathcal A) &= \frac{1}{\mathbb P(I_i^c\cap\mathcal A)} \sum_{j=1}^q p_{x_j}(t)(\mathbb P\circ X^{-1})(\{x_j\})\\ &= \frac{1}{q}\sum_{j=1}^q p_{x_j}(t). \end{align*} Consequently, the upper bound in \eqref{eq:lemineq} is equivalent to \begin{align}\label{eq:stillHolds} 
\sum_{j=1}^q p_{x_j}(t) \leq q e^\varepsilon p_x(t). 
\end{align}
Now suppose that $t\in S_x^\delta$ and, more specifically, that $t$ belongs to exactly one of the bad regions, say $t\in S^\delta_{x,x_q}$ and $t\notin\bigcup_{j=1}^{q-1}S^\delta_{x,x_j}$.
For $j=1,\ldots,q-1$, \autoref{thm:pointwisepDP} gives
\begin{align*}
    e^{-\varepsilon}p_x(t)
    \leq p_{x_j}(t)
    \leq e^\varepsilon p_x(t)
\end{align*}
for $\nu$-almost every such $t$. In particular,
\begin{align*}
    \sum_{j=1}^{q-1}p_{x_j}(t)
    \leq (q-1)e^\varepsilon p_x(t).
\end{align*}
Consider the case in which $p_{x_q}(t)>e^\varepsilon p_x(t)$.
Although the pointwise upper bound is then violated for $x_q$, the corresponding bound for the mixture continues to hold if this excess is compensated by the slack contributed by the remaining terms, that is, if
\begin{align*}
    p_{x_q}(t)-e^\varepsilon p_x(t)
    \leq (q-1)e^\varepsilon p_x(t)
    - \sum_{j=1}^{q-1}p_{x_j}(t).
\end{align*}
Indeed, this condition is equivalent to \eqref{eq:stillHolds} and therefore establishes the upper bound in \eqref{eq:lemineq}. The lower bound also holds in this case. For $j=1,\ldots,q-1$, \autoref{thm:pointwisepDP} gives
\begin{align*}
    p_{x_j}(t)\geq e^{-\varepsilon}p_x(t),
\end{align*}
while
\begin{align*}
    p_{x_q}(t)>e^\varepsilon p_x(t)
    \geq e^{-\varepsilon}p_x(t).
\end{align*}
It follows that
\begin{align*}
    \sum_{j=1}^q p_{x_j}(t)
    \geq q e^{-\varepsilon}p_x(t),
\end{align*}
which establishes the lower bound in \eqref{eq:lemineq}. Thus, \eqref{eq:lemineq} may hold even when $t\in S_x^\delta$ and the pointwise upper bound is violated for $x_q$. At the same time, \autoref{lem:probBounds} guarantees that \eqref{eq:lemineq} holds for $\nu$-almost every $t\in\mathcal T\setminus S_x^\delta$. Hence, up to a $\nu$-null set, the actual failure region is contained in $S_x^\delta$, and 
\begin{align*} 
P_x\!\left( \left\{ t\in\mathcal T\mid \eqref{eq:lemineq}\text{ fails at }t \right\} \right) \leq P_x(S_x^\delta). \end{align*} 
Thus, $S_x^\delta$ is a conservative failure region.

We now use \autoref{lem:probBounds} to establish bounds on the posterior-to-prior ratio $\mathscr{R}$ under bounded probabilistic differential privacy.

\begin{theorem}\textbf{}\label{thm:pDPRatioBounds}\newline
    Assume that $\mathbb{P}(I_{i}\cap\mathcal{A})>0$ and $\mathbb{P}(I_{i}^{c}\cap\mathcal{A})>0$, where $I_{i}$ and $\mathcal{A}$ are given by \eqref{eq:inclusionInference} and \eqref{eq:auxiliaryEvent}, respectively. Let  $M$ be a data-release mechanism satisfying $(\varepsilon,\delta)$-probabilistic differential privacy for $\varepsilon,\delta\in\mathbb{R}_{\geq0}$, and assume that its family of probability measures $\left\{P_{x}\right\}_{x\in\mathcal{X}}$ given by \eqref{eq:probmeasure} is dominated by a $\sigma$-finite measure $\nu$. If $x\in X(I_{i}\cap\mathcal{A})$, then
    \begin{equation}\label{eq:pDPBounds}
          \frac{1}{\mathbb{P}(I_{i}\mid\mathcal{A})+(1-\mathbb{P}(I_{i}\mid\mathcal{A}))e^{\varepsilon}}\leq\mathscr{R}(t)
          \leq\frac{1}{\mathbb{P}(I_{i}\mid\mathcal{A})+(1-\mathbb{P}(I_{i}\mid\mathcal{A}))e^{-\varepsilon}}
    \end{equation}
    for $\nu$-almost every $t\in\mathcal{T}\setminus S^{\delta}_{x}$, where
    \begin{align*}
        S^{\delta}_{x}\coloneqq\bigcup_{x'\in X(I_{i}^{c}\cap\mathcal{A})}S^{\delta}_{x,x'}.
    \end{align*}
\end{theorem}

\begin{proof}
First note that 
\begin{align*}
        P_{M,X}(t,I_{i}\cap\mathcal{A})
        =\int_{I_{i}\cap\mathcal{A}}p_{X(\omega)}(t)d\mathbb{P}(\omega)
        =\int_{I_{i}\cap\mathcal{A}}p_{x}(t)d\mathbb{P}(\omega)
        =p_{x}(t)\mathbb{P}(I_{i}\cap\mathcal{A}),
\end{align*}
since $X(\omega)=x$ for all $\omega\in I_i\cap\mathcal A$.
Hence, 
\begin{align*} 
P_{M\mid X}(I_i\mid t,\mathcal A) 
&= \frac{P_{M,X}(t,I_i\cap\mathcal A)} {P_{M,X}(t,\mathcal A)} \\[0.75ex]
&= \frac{P_{M,X}(t,I_i\cap\mathcal A)} {P_{M,X}(t,I_i\cap\mathcal A) +P_{M,X}(t,I_i^c\cap\mathcal A)} \\[0.75ex]
&= \frac{p_x(t)\mathbb P(I_i\cap\mathcal A)} {p_x(t)\mathbb P(I_i\cap\mathcal A) +P_{M\mid X}(t\mid I_i^c\cap\mathcal A) \mathbb P(I_i^c\cap\mathcal A)} \\[0.75ex]
&= \frac{\mathbb P(I_i\mid\mathcal A)} {\mathbb P(I_i\mid\mathcal A) +\mathbb P(I_i^c\mid\mathcal A) \frac{P_{M\mid X}(t\mid I_i^c\cap\mathcal A)} {p_x(t)}} \\[0.75ex]
&\le \frac{\mathbb P(I_i\mid\mathcal A)} {\mathbb P(I_i\mid\mathcal A) +(1-\mathbb P(I_i\mid\mathcal A))e^{-\varepsilon}}, 
\end{align*}
for $\nu$-almost every $t\in\mathcal{T}\setminus S^{\delta}_{x}$, where the final inequality follows from \autoref{lem:probBounds}. The lower bound in \eqref{eq:pDPBounds} is obtained analogously.
    
\end{proof}

As discussed following \autoref{lem:probBounds}, the condition $t\in\mathcal T\setminus S_x^\delta$ is sufficient, but not necessary, for the bounds in \eqref{eq:pDPBounds} to hold. Indeed, these bounds are derived from \eqref{eq:lemineq}, which is guaranteed to hold outside $S_x^\delta$ but may also hold within this region. Consequently, $S_x^\delta$ is a conservative failure region for \eqref{eq:pDPBounds}, and the probability that these bounds fail is at most $P_x(S_x^\delta)$.

Using \autoref{thm:ADPtoPDP}, we can reformulate \autoref{thm:pDPRatioBounds} to hold for approximately differentially private mechanisms.

\begin{corollary}\textbf{}\label{cor:adpPostPriorBounds}\newline
    Assume that $\mathbb{P}(I_{i}\cap\mathcal{A})>0$ and $\mathbb{P}(I_{i}^{c}\cap\mathcal{A})>0$, where $I_{i}$ and $\mathcal{A}$ are given by \eqref{eq:inclusionInference} and \eqref{eq:auxiliaryEvent}, respectively. Let  $M$ be a data-release mechanism satisfying $(\varepsilon',\delta')$-approximate differential privacy for $\varepsilon',\delta'\in\mathbb{R}_{\geq0}$, and assume that its family of probability measures $\left\{P_{x}\right\}_{x\in\mathcal{X}}$ given by \eqref{eq:probmeasure} is dominated by a $\sigma$-finite measure $\nu$. If $x\in X(I_{i}\cap\mathcal{A})$, then
    \begin{equation}\label{eq:ADPBounds}
        \begin{aligned}
            \frac{1}{\mathbb{P}(I_{i}\mid\mathcal{A})+(1-\mathbb{P}(I_{i}\mid\mathcal{A}))e^{\varepsilon}}&\leq\mathscr{R}(t)\leq\frac{1}{\mathbb{P}(I_{i}\mid\mathcal{A})+(1-\mathbb{P}(I_{i}\mid\mathcal{A}))e^{-\varepsilon}},
        \end{aligned}
    \end{equation}
    for $\nu$-almost every $t\in\mathcal{T}\setminus S^{\delta}_{x}$, where 
    \begin{align*}
        S^{\delta}_{x}\coloneqq\bigcup_{x'\in X(I_{i}^{c}\cap\mathcal{A})}S^{\delta}_{x,x'},
    \end{align*} 
    $\delta\in[\delta',1)$, and
    \begin{align*}
        \varepsilon=\log\left(\frac{\delta e^{\varepsilon'
}+\delta'}{\delta-\delta'}\right).
    \end{align*}
\end{corollary}

\begin{proof}
    The corollary follows immediately from \autoref{cor:adptopdp} and \autoref{thm:pDPRatioBounds}.
\end{proof}

The upper and lower bounds in \eqref{eq:pDPBounds} and \eqref{eq:ADPBounds} depend on the prior probability $\mathbb{P}(I_{i}\mid\mathcal{A})$ adopted by the Bayesian adversary. However, \cite{kazan2025bayesian} show that the posterior-to-prior ratio also satisfies the prior-independent bounds
\begin{align}\label{eq:priorIndep}
    e^{-\varepsilon}\leq\mathscr{R}(t)\leq e^{\varepsilon},
\end{align}
for $\nu$-almost every $t\in\mathcal{T}\setminus S^{\delta}_{x}$. 
As these bounds define a wider interval, the probability that $\mathscr{R}(t)$ falls outside them is not greater than for the corresponding prior-dependent bounds.

Quantifying when the posterior-to-prior ratio bounds fail is important when assessing whether the release of differentially private data poses an unacceptable disclosure risk. Finding a general expression for $P_{x}(S^{\delta}_{x})$ is elusive. However, an upper limit can be obtained by exploiting the subadditivity of $P_{x}$, namely, 
\begin{align*}
    P_{x}(S^{\delta}_{x})=P_{x}\left(\bigcup_{x'\in X(I_{i}^{c}\cap\mathcal{A})}S^{\delta}_{x,x'}\right)\leq\sum_{x'\in X(I_{i}^{c}\cap\mathcal{A})}P_{x}(S^{\delta}_{x,x'})=q\delta,
\end{align*}
where $q=|X(I_{i}^{c}\cap\mathcal{A})|$. In our setting, $q=n-m$ may be very large, rendering the upper limit vacuous. As the sets $S^{\delta}_{x,x'}$, $x'\in X(I_{i}^{c}\cap\mathcal{A})$, may overlap, $P_{x}(S^{\delta}_{x})$ can be significantly smaller than $q\delta$. Nevertheless, for the Gaussian mechanism an exact expression for $P_{x}(S^{\delta}_{x})$ can be derived, as shown in the following section.

%% file: failingRisk.tex
To quantify the probability of violating the bounds in \eqref{eq:pDPBounds} and \eqref{eq:ADPBounds}, it remains to characterise $P_x(S_x^\delta)$. Since a closed-form expression for $P_x(S^\delta_{x,x'})$ was derived for the Gaussian mechanism in \autoref{sec:relationship}, we can now use it to characterise $P_x(S^\delta_x)$. Recall that
\begin{align*}
    S^{\delta}_{x}=\bigcup_{x'\in X(I_{i}^{c}\cap\mathcal{A})}S^{\delta}_{x,x'},
\end{align*}
where $I_{i}$ is given by \eqref{eq:inclusionInference}, $\mathcal{A}$ is given by \eqref{eq:auxiliaryEvent}, and $x\in X(I_{i}\cap\mathcal{A})$. Note that $X(I_{i}\cap\mathcal{A})= \{ x^{-}\cup\{i\}\}$, while $X(I^{c}_{i}\cap\mathcal{A})$ contains $p\coloneqq n-m+1$ elements, say $\{x_{1},x_{2},\ldots,x_{p}\}$.
Then
\begin{align*}
    P_{x}(S^{\delta}_{x})
    =P_{x}\left(\bigcup_{i=1}^{p} S^{\delta}_{x,x_{i}}\right)=P_{x}\left(\bigcup_{i=1}^{p}\left(S^{\delta}_{x,x_{i}}\setminus\bigcup_{j=1}^{i-1}S^{\delta}_{x,x_{j}}\right)\right),
\end{align*}
with the convention that $\bigcup_{j=1}^{0}S^{\delta}_{x,x_{j}}=\emptyset$.
Since the outer union is taken over disjoint sets, the additivity of $P_{x}$ implies that 
\begin{align}\label{eq:disjointFamily}
    P_{x}\left(\bigcup_{i=1}^{p}\left(S^{\delta}_{x,x_{i}}\setminus\bigcup_{j=1}^{i-1}S^{\delta}_{x,x_{j}}\right)\right)&=\sum_{i=1}^{p}P_{x}\left(S^{\delta}_{x,x_{i}}\setminus\bigcup_{j=1}^{i-1}S^{\delta}_{x,x_{j}}\right)\nonumber\\
    &=\sum_{i=1}^{p}P_{x}\left(S^{\delta}_{x,x_{i}}\cap\bigcap_{j=1}^{i-1}\left(S^{\delta}_{x,x_{j}}\right)^{c}\right),
\end{align}
where, by convention,  $\bigcap_{j=1}^{0}\left(S^{\delta}_{x,x_{j}}\right)^{c}=\mathcal{T}$.

To facilitate the analysis of the summands above, we first introduce some notation. Recall the Gaussian mechanism $M$ described by 
\begin{align*}
    M(x,u)\coloneqq f(x)+\sigma\Phi^{-1}(u),
\end{align*}
where $f:\mathcal{X}\to A\subset\mathbb{R}^{q}$ is a deterministic function and the entries of $\Phi^{-1}:[0,1]\to\mathbb{R}^{q}$ are inverse standard normal distribution functions. Define $m_{j}\coloneqq\frac{1}{2}(f(x)+f(x_{j}))$, $\Delta f_{j}\coloneqq f(x)-f(x_{j})$, $\Delta_{j}\coloneqq\norm{\Delta f_{j}}_{2}$, and $\tilde{u}_{j}\coloneqq\Delta f_{j}/\Delta_{j}$ for each $j=1,2,\ldots,p$ with $f(x)\neq f(x_{j})$. Note, for $f(x)=f(x_{j})$, we have $S^{\delta}_{x,x_{j}}=\emptyset$ as shown in \autoref{sec:relationship}. Using this notation, we have from \eqref{eq:badRegionShape} that
\begin{align*}
    S^{\delta}_{x,x_{j}}=\left\{v\in\mathbb{R}^{q}\;\middle|\; |\langle v-m_{j},\tilde{u}_{j}\rangle|>\frac{\varepsilon\sigma^{2}}{\Delta_{j}}\right\}.
\end{align*}
Define the random variables $V(u)\coloneqq M(x,u)$ and $W_{j}(u)\coloneqq\langle V(u)-m_{j},\tilde{u}_{j}\rangle$, $j=1,2,\ldots,p$. For $U\sim \mathrm{Unif}[0,1]$, we have that $V(U)\sim N(f(x),\sigma^{2}I_{q})$, and in turn
\begin{align*}
    \mathbb{E}\left[W_{j}\right]&=\mathbb{E}\left[\langle V-m_{j},\tilde{u}_{j}\rangle\right]=\langle f(x)-m_{j},\tilde{u}_{j}\rangle=\frac{\Delta_{j}}{2},\\
    \text{Var}\left[W_{j}\right]&=\text{Var}\left[\langle V-m_{j},\tilde{u}_{j}\rangle\right]
    =\sigma^{2}.
\end{align*}
We can then define the standardised version of $W_{j}$,
\begin{align*}
    N_{j}\coloneqq\frac{W_{j}-\Delta_{j}/2}{\sigma}.
\end{align*}
Combining the above, we can reexpress $P_{x}(S^{\delta}_{x,x_{j}})$ in terms of the Lebesgue measure, 
\begin{align*}
    P_{x}(S^{\delta}_{x,x_{j}})&=P_{x}\left(\left\{v\in\mathbb{R}^{q}\;\middle|\; |\langle v-m_{j},\tilde{u}_{j}\rangle|>\frac{\varepsilon\sigma^{2}}{\Delta_{j}}\right\}\right)\\[0.75ex]
    &=\lambda\left(\left\{u\in[0,1]\;\middle|\;|\langle V(u)-m_{j},\tilde{u}_{j}\rangle|>\frac{\varepsilon\sigma^{2}}{\Delta_{j}}\right\}\right)\\[0.75ex]
    &=\lambda\left(\left\{|\langle V-m_{j},\tilde{u}_{j}\rangle|>\frac{\varepsilon\sigma^{2}}{\Delta_{j}}\right\}\right)\\[0.75ex]
    &=\lambda\left(\left\{\left|\frac{\Delta_{j}}{2}+\sigma N_{j}\right|>\frac{\varepsilon\sigma^{2}}{\Delta_{j}}\right\}\right).
\end{align*}
Using this representation, consider now the summand in \eqref{eq:disjointFamily},
\begin{align*}
    P_{x}\left(S^{\delta}_{x,x_{i}}\cap\bigcap_{j=1}^{i-1}\left(S^{\delta}_{x,x_{j}}\right)^{c}\right)
    =&\lambda\left(\left\{\left|\frac{\Delta_{i}}{2}+\sigma N_{i}\right|>\frac{\varepsilon\sigma^{2}}{\Delta_{i}}\right\}\cap\bigcap_{j=1}^{i-1}\left\{\left|\frac{\Delta_{j}}{2}+\sigma N_{j}\right|\leq\frac{\varepsilon\sigma^{2}}{\Delta_{j}}\right\}\right)\\[0.75ex]
    =&\lambda\left(\left\{\frac{\Delta_{i}}{2}+\sigma N_{i}<-\frac{\varepsilon\sigma^{2}}{\Delta_{i}}\right\}\cap\bigcap_{j=1}^{i-1}\left\{-\frac{\varepsilon\sigma^{2}}{\Delta_{j}}\leq\frac{\Delta_{j}}{2}+\sigma N_{j}\leq\frac{\varepsilon\sigma^{2}}{\Delta_{j}}\right\}\right)\\
    & \quad+\lambda\left(\left\{\frac{\Delta_{i}}{2}+\sigma N_{i}>\frac{\varepsilon\sigma^{2}}{\Delta_{i}}\right\}\cap\bigcap_{j=1}^{i-1}\left\{-\frac{\varepsilon\sigma^{2}}{\Delta_{j}}\leq\frac{\Delta_{j}}{2}+\sigma N_{j}\leq\frac{\varepsilon\sigma^{2}}{\Delta_{j}}\right\}\right)\\[0.75ex]
    =&\lambda\left(\left\{N_{i}<-\frac{\varepsilon\sigma}{\Delta_{i}}-\frac{\Delta_{i}}{2\sigma}\right\}\cap\bigcap_{j=1}^{i-1}\left\{-\frac{\varepsilon\sigma}{\Delta_{j}}-\frac{\Delta_{j}}{2\sigma}\leq N_{j}\leq\frac{\varepsilon\sigma}{\Delta_{j}}-\frac{\Delta_{j}}{2\sigma}\right\}\right)\\
    &  \quad +\lambda\left(\left\{N_{i}>\frac{\varepsilon\sigma}{\Delta_{i}}-\frac{\Delta_{i}}{2\sigma}\right\}\cap\bigcap_{j=1}^{i-1}\left\{-\frac{\varepsilon\sigma}{\Delta_{j}}-\frac{\Delta_{j}}{2\sigma}\leq N_{j}\leq\frac{\varepsilon\sigma}{\Delta_{j}}-\frac{\Delta_{j}}{2\sigma}\right\}\right).
\end{align*}
For $j=1,2,\ldots,p$, define
\begin{align*}
    t_{j}^{(1)}\coloneqq-\frac{\varepsilon\sigma}{\Delta_{j}}-\frac{\Delta_{j}}{2\sigma}, \quad t_{j}^{(2)}\coloneqq\frac{\varepsilon\sigma}{\Delta_{j}}-\frac{\Delta_{j}}{2\sigma}.
\end{align*}
We then obtain the simplified expression
\begin{equation}\label{eq:simplifiedExpression}
\begin{split}
    P_{x}\left(S^{\delta}_{x,x_{i}}\cap\bigcap_{j=1}^{i-1}\left(S^{\delta}_{x,x_{j}}\right)^{c}\right)&=\lambda\left(\left\{N_{i}<t_{i}^{(1)}\right\}\cap\bigcap_{j=1}^{i-1}\left\{t_{j}^{(1)}\leq N_{j}\leq t_{j}^{(2)}\right\}\right)\\
    &\quad +\lambda\left(\left\{N_{i}>t^{(2)}_{i}\right\}\cap\bigcap_{j=1}^{i-1}\left\{t_{j}^{(1)}\leq N_{j}\leq t_{j}^{(2)}\right\}\right).
\end{split}
\end{equation}
Each of the two summands in \eqref{eq:simplifiedExpression} is a probability under the joint distribution of $(N_{1},N_{2},\ldots,N_{i})$, which is a zero-mean multivariate normal distribution. Evaluating these probabilities therefore reduces to determining the corresponding covariance matrix. For $i,j=1,2,\ldots,p$, define 
\begin{align*} 
\rho_{i,j} &\coloneqq \mathrm{Cov}\left[N_i,N_j\right] 
= \frac{1}{\sigma^2}\mathrm{Cov}\left[W_i,W_j\right] 
= \frac{1}{\sigma^2}\mathrm{Cov}\left[\langle V,\tilde u_i\rangle,\langle V,\tilde u_j\rangle\right] 
\stackrel{\dagger}{=} \langle \tilde u_i,\tilde u_j\rangle. 
\end{align*}
where $\dagger$ follows from $\langle V,\tilde{u}_{i}\rangle=\tilde{u}_{i}^{\top}V$ and $\langle V,\tilde{u}_{j}\rangle=\tilde{u}_{j}^{\top}V$. Let $\Sigma_{i}^{(1)}$ denote the $i\times i$ covariance matrix of $(N_1, \ldots, N_i)$, i.e.,
\begin{align*}
    \Sigma_{i}^{(1)}\coloneqq\begin{bmatrix}
        1 & \rho_{1,2} & \cdots & \rho_{1,i} \\
        \rho_{2,1} & 1 & \cdots & \rho_{2,i} \\
        \vdots & \vdots & \ddots & \vdots \\
        \rho_{i,1} & \rho_{i,2} & \cdots & 1
    \end{bmatrix}.
\end{align*}
Now the first term on the right hand side in \eqref{eq:simplifiedExpression} can be rewritten:
\begin{align*}
    \lambda\left(\left\{N_{i}<t_i^{(1)}\right\}\cap\bigcap_{j=1}^{i-1}\left\{t_{j}^{(1)}\leq N_{j}\leq t_{j}^{(2)}\right\}\right)
      =\Phi\left(\left[t_{1}^{(1)},t_{1}^{(2)}\right],\left[t_{2}^{(1)},t_{2}^{(2)}\right],\ldots,\left[t_{i-1}^{(1)},t_{i-1}^{(2)}\right],t_i^{(1)};\Sigma_{i}^{(1)}\right),
\end{align*}
where $\Phi(\cdot,\dots,\cdot;\Sigma_{i}^{(1)})$ denotes the cumulative distribution function of the zero-mean multivariate normal distribution with covariance matrix $\Sigma_{i}^{(1)}$. For the second term of the right hand side in \eqref{eq:simplifiedExpression}, we exploit that \begin{align*}
    \lambda\left(\left\{N_{i}>t_{i}^{(2)}\right\}\right)=\lambda\left(\left\{N_{i}<-t_{i}^{(2)}\right\}\right),
\end{align*}
and obtain a similar expression:
\begin{align*}
    &\lambda\left(\left\{N_{i}>t^{(2)}_{i}\right\}\cap\bigcap_{j=1}^{i-1}\left\{t_{j}^{(1)}\leq N_{j}\leq t_{j}^{(2)}\right\}\right)\\
    &\quad=\Phi\left(\left[t_{1}^{(1)},t_{1}^{(2)}\right],\left[t_{2}^{(1)},t_{2}^{(2)}\right],\ldots,\left[t_{i-1}^{(1)},t_{i-1}^{(2)}\right],-t_i^{(2)};\Sigma_{i}^{(2)}\right),
\end{align*}
where
\begin{align*}
    \Sigma_{i}^{(2)}\coloneqq\begin{bmatrix}
        1 & \rho_{1,2} & \cdots & \rho_{1,i-1} & -\rho_{1,i} \\
        \rho_{2,1} & 1 & \cdots & \rho_{2,i-1} & -\rho_{2,i} \\
        \vdots & \vdots & \ddots & \vdots & \vdots \\
        \rho_{i-1,1} & \rho_{i-1,2} & \cdots & 1 & -\rho_{i-1,i}\\
       - \rho_{i,1} & -\rho_{i,2} & \cdots & -\rho_{i,i-1} & 1
    \end{bmatrix}.
\end{align*}
Substituting these expressions into \eqref{eq:simplifiedExpression} yields 
\begin{align*}
    P_{x}\left(S^{\delta}_{x,x_{i}}\cap\bigcap_{j=1}^{i-1}\left(S^{\delta}_{x,x_{j}}\right)^{c}\right)&=\sum_{j=1}^{2}\Phi\left(\left[t_{1}^{(1)},t_{1}^{(2)}\right],\left[t_{2}^{(1)},t_{2}^{(2)}\right],\ldots,\left[t_{i-1}^{(1)},t_{i-1}^{(2)}\right],(-1)^{j-1}t_i^{(j)};\Sigma_{i}^{(j)}\right).
\end{align*}
Combining this with \eqref{eq:disjointFamily}, we obtain the following exact expression for the probability that the Gaussian mechanism produces an output in the $S_x^\delta$:
\begin{align}\label{eq:exactGaussianFormula}
    P_{x}(S^{\delta}_{x})=\sum_{i=1}^{p}\sum_{j=1}^{2}\Phi\left(\left[t_{1}^{(1)},t_{1}^{(2)}\right],\left[t_{2}^{(1)},t_{2}^{(2)}\right],\ldots,\left[t_{i-1}^{(1)},t_{i-1}^{(2)}\right],(-1)^{j-1}t_i^{(j)};\Sigma_{i}^{(j)}\right).
\end{align}
Again, since $S_x^\delta$ is a conservative failure region, this expression provides only an upper limit to the failure probability of the posterior-to-prior ratio bounds \eqref{eq:pDPBounds} and \eqref{eq:ADPBounds}.

%% file: empirical.tex
Although $P_x(S_x^\delta)$ provides a conservative upper limit on the probability that $\mathscr{R}$ violates the bounds in \eqref{eq:pDPBounds} and \eqref{eq:ADPBounds}, the tightness of this limit is not immediately clear. For the Gaussian mechanism, \eqref{eq:exactGaussianFormula} gives an exact expression for $P_x(S_x^\delta)$, which we compare with the empirical failure probability obtained via Monte Carlo simulation. This comparison allows us to assess how conservative the limit is in practice. We focus on the failure probability under probabilistic differential privacy, as the corresponding analysis under approximate differential privacy differs only in the choice of $(\varepsilon,\delta)$-parameters by the results in \autoref{sec:relationship}.

To carry out this comparison, we first calibrate the Gaussian mechanism to satisfy a prescribed $(\varepsilon,\delta)$-probabilistic differential privacy guarantee. This requires selecting the noise parameter $\sigma$ such that the probability mass of the bad region is approximately $\delta$.
Recall from \eqref{eq:exactBadSize} that for any $x,x'\in\mathcal{X}$ with $d(x,x')=1$, 
\begin{align*} 
    P_{x}(S^{\delta}_{x,x'})= \Phi\left(-\frac{\varepsilon\sigma}{\Delta}+\frac{\Delta}{2\sigma}\right) +\Phi\left(-\frac{\varepsilon\sigma}{\Delta}-\frac{\Delta}{2\sigma}\right). 
\end{align*} 
Hence, to obtain exactly $(\varepsilon,\delta)$-probabilistic differential privacy, we choose $\sigma>0$ by minimising
\begin{align}\label{eq:lagrange}
 \mathcal{L}(\sigma)\coloneqq\left|\Phi\left(-\frac{\varepsilon\sigma}{\Delta}+\frac{\Delta}{2\sigma}\right)+\Phi\left(-\frac{\varepsilon\sigma}{\Delta}-\frac{\Delta}{2\sigma}\right)-\delta\right|.
\end{align}
This calibration is similar to the approach of \cite{balle2018improving}, except that their minimisation yields the optimal noise parameter for approximate differential privacy.

\subsubsection*{Simulation setup}

To assess the tightness of the theoretical bound, we consider the following two-level simulation setup. In the outer loop, we simulate $N_{1}$ datasets of size $m$ from a population $P =\{1,2,\ldots,n\}$ of size $n>m$. For each sampled dataset, we calculate the theoretical upper bound $P_x(S_x^\delta)$. In the inner loop, we independently apply the Gaussian mechanism with probabilistic differential privacy $N_2$ times to each dataset, taking $f(x)$ to be the vector of attribute values of the individuals in $x$. This generates a total of $N_1N_2$ noisy datasets. For each realisation, we calculate the posterior-to-prior ratio and record whether it violates the bounds in \eqref{eq:pDPBounds}. The empirical failure probability, denoted $\overline{F}$, is then estimated as the proportion of the $N_1N_2$ posterior-to-prior ratios outside these bounds. Since the theoretical upper limit depends on the specific dataset realisation, we compare its mean across the $N_1$ datasets, denoted $\overline{T}$, with $\overline{F}$. 

For each dataset $x=\{x_1,\ldots,x_m\}$ sampled in the outer loop from the population $P$, we treat $x^-=x \setminus \{x_m\}$ as auxiliary information available to the adversary. The adversary therefore knows all but one member of the dataset and seeks to infer the identity of the remaining individual from among the individuals in $P$ not contained in $x^-$. Accordingly, the candidate set is \begin{align*} \mathcal C=P\setminus x^-, \end{align*} and the true remaining individual is $x_m\in\mathcal C$. After observing the noisy output of the Gaussian mechanism, the adversary updates the prior inclusion probability of $x_m$. The resulting posterior-to-prior ratio is used to determine whether the bounds in \eqref{eq:pDPBounds} are violated. The following procedure outlines the simulation framework used to compute $\overline{T}$ and $\overline{F}$.

\begin{algorithm}[H]
\caption*{\textbf{Procedure:} Monte Carlo estimation of failure probability}
\begin{algorithmic}[1]
\State \textbf{Input:} Privacy parameters: $\varepsilon > 0$, $\delta \in ]0,1]$; simulation parameters $n,m,N_{1},N_{2}\in\mathbb{N}$
\Statex \rule[0.5ex]{\linewidth}{0.2pt}
\State Set population $P\gets\{1,2,\ldots,n\}$
\State Compute $\sigma\gets\arg\min_{\tilde{\sigma}>0}\mathcal{L}(\tilde{\sigma})$
\State Initialise $N_1\times N_2$ matrix $F$
\State Initialise size $N_{1}$ vector $T$
\State
\For{$j=1,2,\ldots,N_{1}$}
    \State Sample attributes $\phi_i^{(j)}\sim\mathrm{Unif}[0,1]$ independently for all $i\in P$
    \State Sample dataset $x\subset P$ uniformly without replacement with $|x| = m$
    \State Enumerate the elements of $x$ as $x=\{x_1,\ldots,x_m\}$
    \State Define the set of auxiliary information $x^{-}=x\setminus \{x_{m}\}$
    \State Compute $T_{j}\gets P_{x}(S^{\delta}_{x})$ using \eqref{eq:exactGaussianFormula}
    \State
    \For{$k=1,2,\ldots,N_{2}$}
        \State Generate noisy dataset $t$, where $t_\ell \gets \phi^{(j)}_{x_\ell}+N(0,\sigma^2)$ for $\ell=1,\ldots,m$
        \State Compute $\mathscr{R}(t)$ according to \eqref{eq:sim_ratio} for $I_{i}$, where $i=x_{m}$, and $\mathcal{A}=\{\omega\in\Omega\mid x^{-}\subset X(\omega)\}$
        \If{$\mathscr{R}(t)$ satisfies \eqref{eq:pDPBounds}}
            \State $F_{j,k}\gets 0$
        \Else
            \State $F_{j,k}\gets 1$
        \EndIf
    \EndFor
\EndFor
\State
\State \Return  $\overline{T}\coloneqq\frac{1}{N_{1}}\sum_{j=1}^{N_{1}}T_{j}$;\quad $\overline{F}\coloneqq\frac{1}{N_{1}N_{2}}\sum_{j=1}^{N_{1}}\sum_{k=1}^{N_{2}}F_{j,k}$
\end{algorithmic}
\end{algorithm}

The posterior-to-prior ratio in Line 16 of the simulation procedure is obtained by first deriving the posterior probability in \eqref{eq:postpriorratio}.  Applying Bayes' rule yields
\begin{align*}
    P_{X\mid M}(I_{i}\mid t,\mathcal{A})&\propto P_{M\mid X}(t\mid I_{i}\cap\mathcal{A})\mathbb{P}(I_{i}\cap\mathcal{A})\\
    &\propto P_{M\mid X}(t\mid I_{i}\cap\mathcal{A})\mathbb{P}(I_{i}\mid\mathcal{A})\\
    &\propto P_{M\mid X}(t\mid I_{i}\cap\mathcal{A}).
\end{align*}
The last proportionality follows from the fact that the prior inclusion probability $\mathbb{P}(I_{i}\mid\mathcal{A})$ is uniform over the set of candidates $\mathcal{C}$. In particular, $\mathbb{P}(I_{i}\mid\mathcal{A})=1/|\mathcal{C}|$. Since the Gaussian mechanism at Line 15 adds independent $N(0,\sigma^2)$ noise to each attribute, the conditional density of $t$ given a candidate $i\in\mathcal C$ is 
\begin{align}\label{eq:stemsfrom}
    P_{M\mid X}(t\mid I_{i}\cap\mathcal{A})\propto\exp\left(-\frac{1}{2\sigma^2}\left(\left(t_{m}-\phi^{(j)}_i\right)^2+\sum_{\ell=1}^{m-1}\left(t_{\ell}-\phi^{(j)}_{x_{\ell}}\right)^2\right)\right)\eqqcolon\gamma(i).  
\end{align}
Since \(\gamma(i)\) is proportional to the posterior probability \(P_{X\mid M}(I_i\mid t,\mathcal A)\) for all \(i\in\mathcal C\), normalising over \(\mathcal C\) gives \[ P_{X\mid M}(I_i\mid t,\mathcal A) = \frac{\gamma(i)} {\sum_{q\in\mathcal C}\gamma(q)}. \]
Finally, dividing by the prior, we get the posterior-to-prior ratio
\begin{align}\label{eq:sim_ratio}
    \mathscr{R}(t)=\frac{P_{X\mid M}(I_{i}\mid t,\mathcal{A})}{\mathbb{P}(I_{i}\mid\mathcal{A})}=\frac{\gamma(i)}{\sum_{q\in\mathcal{C}}\gamma(q)}|\mathcal{C}|.
\end{align}

\subsubsection*{Properties of the Monte Carlo estimators}

Before presenting the simulation results, we investigate the theoretical properties of the Monte Carlo estimator. Let $Z_j$ denote all randomness generated in the $j$th iteration of the outer loop, including the attribute vector $\phi^{(j)}$ and the sampled dataset $x^{(j)}$. We assume that $Z_1,\ldots,Z_{N_1}$ are i.i.d. For $j=1,\ldots,N_1$ and $k=1,\ldots,N_2$, let $F_{j,k}$ indicate whether the posterior-to-prior ratio obtained in the $k$th inner-loop iteration violates the bounds in \eqref{eq:pDPBounds}. Conditional on $Z_j$, we assume that $F_{j,1},\ldots,F_{j,N_2}$ are independent and satisfy 
\begin{align*} 
    F_{j,k}\mid Z_j \sim \operatorname{Bernoulli}\!\left(\rho(Z_j)\right), 
\end{align*} 
where $\rho(Z_j)$ denotes the conditional failure probability associated with the $j$th outer-loop realisation. The corresponding unconditional failure probability is 
\begin{align*} 
    \rho\coloneqq\mathbb E\!\left[\rho(Z_j)\right]. 
\end{align*} 
Define 
\begin{align*} 
    \overline F_j \coloneqq \frac{1}{N_2}\sum_{k=1}^{N_2}F_{j,k} \qquad\text{and}\qquad \overline F \coloneqq \frac{1}{N_1}\sum_{j=1}^{N_1}\overline F_j. 
\end{align*} 
By the law of total expectation, 
\begin{align*} 
    \mathbb E[\overline F_j] &= \mathbb E\left[ \mathbb E[\overline F_j\mid Z_j] \right]\\ &= \mathbb E\left[ \frac{1}{N_2} \sum_{k=1}^{N_2} \mathbb E[F_{j,k}\mid Z_j] \right]\\ &= \mathbb E[\rho(Z_j)]\\ &= \rho. 
\end{align*} 
Consequently, 
\begin{align*} 
    \mathbb E[\overline F] = \frac{1}{N_1} \sum_{j=1}^{N_1}\mathbb E[\overline F_j] = \rho. 
\end{align*} 
Hence, both $\overline F_j$ and $\overline F$ are unbiased estimators of $\rho$. For a fixed outer-loop realisation, however, the conditional law of large numbers gives 
\begin{align*} 
    \overline F_j = \frac{1}{N_2}\sum_{k=1}^{N_2}F_{j,k} \to \rho(Z_j) \qquad\text{a.s. as }N_2\to\infty. 
\end{align*} 
Thus, increasing $N_2$ makes $\overline F_j$ consistent for the realisation-specific failure probability $\rho(Z_j)$, but not for the unconditional failure probability $\rho$, since 
\begin{align*} 
    \operatorname{Var}[\rho(Z_j)]>0. 
\end{align*} 
By contrast, averaging over independent outer-loop realisations yields a consistent estimator of $\rho$. Since $\overline F_1,\ldots,\overline F_{N_1}$ are i.i.d., 
\begin{align*} 
    \operatorname{Var}[\overline F] &= \operatorname{Var}\left[ \frac{1}{N_1}\sum_{j=1}^{N_1}\overline F_j \right]\\ &= \frac{\operatorname{Var}[\overline F_j]}{N_1} \to 0 \qquad\text{as }N_1\to\infty. 
\end{align*} 
Together with the unbiasedness of $\overline F$, this implies that $\overline F$ converges to $\rho$ in probability. Therefore, $\overline F$ is a consistent estimator of the unconditional failure probability as $N_1\to\infty$.

\subsubsection*{Results}

\textit{Source code and plots are available at \href{https://github.com/janreiter793/ratio_bounds_failure_rate}{\textcolor{blue}{https://github.com/janreiter793/ratio\_bounds\_failure\_rate}}.}

We investigate the behaviour of the Monte Carlo estimators through parameter sweeps over $\varepsilon$, $m$, and $\delta$, while fixing $n = 100$, $N_{1}=1,000$, and $N_{2}=1,000,000$. We have limited the analysis to a population size of $n=100$ due to high computational demands for larger $n$; however, the points of the analysis should still apply to $n>100$. The ranges are given in the following table
\begin{table}[H]
    \centering
    \begin{tabular}{cl}
        \toprule
        Variable & Range \\
        \midrule
        $\varepsilon$ & $[0.2,2.0]$ \\
        $m$ & $\{10,\ldots,99\}$ \\
        $\delta$ & $[10^{-1.1},1]$ \\
        \bottomrule
    \end{tabular}
\end{table}
The resulting estimates are displayed in \autoref{fig:epsilon_sweep}, \autoref{fig:m_sweep}, and \autoref{fig:delta_sweep}.

\begin{figure}[H]
    \centering
    \resizebox{0.75\textwidth}{!}{
        \input{figures/epsilon_sweep_plot_v3}
    }
    \caption{Empirical estimates and standard deviations of the failure probability $\overline{F}$ (blue) and the corresponding mean theoretical upper limit $\overline{T}$ (red) for $40$ equidistant values of $\varepsilon$ in the interval $[0.2,2.0]$. Results are shown for $n=100$, $m=90$, and $\delta=10^{-1.1}$. Estimates are based on $N_1 = 1,000$ attribute samples and $N_2 = 1,000,000$ simulations per attribute sample.
    }\label{fig:epsilon_sweep}
\end{figure}
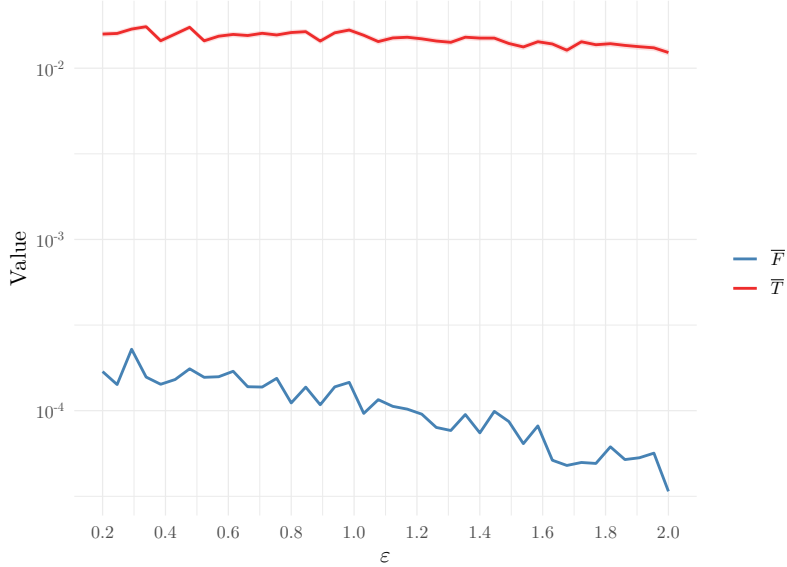

As shown in \autoref{fig:epsilon_sweep}, the average theoretical upper limit $\overline{T}$ consistently exceeds the empirical failure probability $\overline{F}$ by approximately two orders of magnitude throughout the interval $\varepsilon\in[0.2,2.0]$. While $\overline{F}$ decreases noticeably as $\varepsilon$ increases, $\overline{T}$ remains visibly constant around $1.5\times10^{-2}$. The standard deviations of $\overline{F}$ and $\overline{T}$ are within the ranges of $[5.82\cdot 10^{-12},1.51\cdot10^{-11}]$ and $[5.16\cdot10^{-4},6.22\cdot10^{-4}]$, respectively, and are therefore negligible on the scale of \autoref{fig:epsilon_sweep}. 
Averaged over the chosen values of $\varepsilon$, the mean difference between $\overline{T}$ and $\overline{F}$ is $1.48\times10^{-2}$, while the mean value of $\overline{T}$ is $1.49\times10^{-2}$. This suggests that the theoretical upper bound is highly conservative in practice, with the empirical failure probability contributing only a small fraction of the bound across the entire range of $\varepsilon$ considered.

Although $\overline{T}$ is not visibly decreasing in \autoref{fig:epsilon_sweep}, a linear regression of $\overline{T}$ on $\varepsilon$ yields a negative slope of $-1.83\cdot10^{-3}$ ($p$-value $=5.22\cdot10^{-10}$). Hence, both $\overline{F}$ and $\overline{T}$ exhibit a decreasing trend over the considered values of $\varepsilon$. The decrease in $\overline{F}$ indicates that the posterior-to-prior ratio $\mathscr{R}$ satisfies the bounds \eqref{eq:pDPBounds} more often as the privacy budget increases. This is consistent with the widening of the admissible interval as $\varepsilon$ increases. Indeed, considering \eqref{eq:pDPBounds}, we see that the bounds expand with $\varepsilon$:
\begin{align*}
    \left[\frac{1}{\mathbb{P}(I_{i}\mid\mathcal{A})+(1-\mathbb{P}(I_{i}\mid\mathcal{A}))e^{\varepsilon}}, \,
    \frac{1}{\mathbb{P}(I_{i}\mid\mathcal{A})+(1-\mathbb{P}(I_{i}\mid\mathcal{A}))e^{-\varepsilon}}\right]\to\left[0,\frac{1}{\mathbb{P}(I_{i}\mid\mathcal{A})}\right],
\end{align*}
as $\varepsilon\to\infty$. Importantly, the limiting interval remains finite. At the same time, increasing $\varepsilon$ weakens the privacy guarantee and results in less perturbed data. Consequently, the posterior distribution may deviate more substantially from the prior distribution, potentially increasing the variability of the ratio $\mathscr{R}$. The asymptotic behaviour of $\overline{F}$ thus depends on the interaction between the widening of the admissible interval and the changing distribution of $\mathscr{R}$. Although the simulations cover only a finite range of $\varepsilon$ and therefore do not determine this asymptotic behaviour, we conjecture that $\overline{F}$ may converge to a strictly positive limit or eventually exhibit qualitatively different behaviour if the increasing variability of $\mathscr{R}$ outweighs the widening of the admissible interval.

\autoref{fig:m_sweep} shows $\overline{T}$ and $\overline{F}$ for observed dataset sizes from $m=10$ to $m = 99$. We limit the range to $m=99$, since $m=n$ would result in the candidate set $\mathcal{C}$ containing only $x_m$. Consequently, $\mathbb{P}(I_i\mid\mathcal{A})=1$ for $i\in\mathcal{C}$ and, likewise, $P_{X\mid M}(I_i\mid t,\mathcal{A})=1$, resulting in $\mathscr{R}(t)=1$. Since $\mathscr{R}(t)=1$ always satisfies \eqref{eq:pDPBounds}, the empirical failure probability would be zero.

There are two key observations from \autoref{fig:m_sweep}. First, for most values of $m$ there is a substantial gap between $\overline{F}$ and $\overline{T}$. As $m$ increases towards $99$, $\overline{T}$ tends to decrease while $\overline{F}$ tends to increase, with the two quantities coinciding at $m=99$. The downward trend of $\overline{T}$ is explained by the decreasing number of neighbours of $x$. Specifically, since
\begin{align*}
    S^{\delta}_{x}=\bigcup_{x'\in X(I^{c}_{i}\cap\mathcal{A})}S^{\delta}_{x,x'},
\end{align*}
the set $S^{\delta}_{x}$ is formed from a union of fewer and fewer sets $S^{\delta}_{x,x'}$ as $m$ approaches $n-1$. As a result, the size of $S^{\delta}_{x}$ decreases, causing $P_{x}(S^{\delta}_{x})$ to decrease, which is reflected in the downward trend of $\overline{T}$.
The second observation concerns the shape of this downward trend. In particular, $\overline{T}$ steadily decreases until approximately $m=70$, after which the decline becomes markedly steeper, resulting in the hockey-stick shape of the curve. As $m$ increases, the number of neighbours to the observed dataset $x$ decreases, reducing the number of bad regions contributing to $S^{\delta}_{x}$. The non-linear shape of $\overline{T}$ further suggests that the regions $S^{\delta}_{x, x'}$ are overlapping or differing in size.
Although these observations are based on $n=100$, we hypothesise that the trends hold for any population size $n$. Further, the standard deviations for $\overline{T}$ and $\overline{F}$ in \autoref{fig:m_sweep} are within $[2.96\cdot10^{-4},7.44\cdot10^{-4}]$ and $[6.03\cdot10^{-12},5.64\cdot10^{-11}]$, respectively, and are therefore negligible in the scale of the figure. 

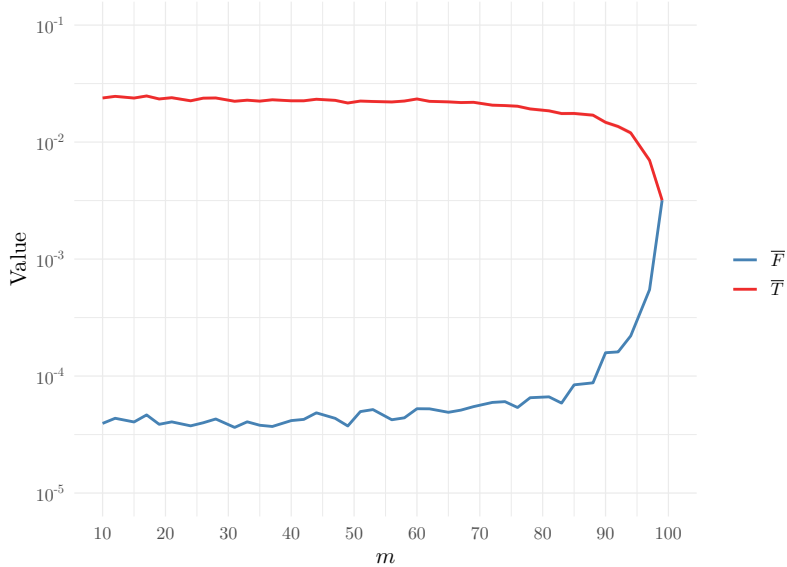
\begin{figure}[htbp]
    \centering
    \resizebox{0.75\textwidth}{!}{
        \input{figures/m_sweep_plot_v3}
    }
    \caption{Empirical estimates and standard deviations of the failure probability $\overline{F}$ (blue) and the corresponding mean theoretical upper limit $\overline{T}$ (red) for $40$ equally distributed integer values of $m$ from $m=10$ to $m=99$. Results are shown for $n=100$, $\varepsilon=1$, and $\delta=10^{-1.1}$. Estimates are based on $N_1=1,000$ attribute samples and $N_2=1,000,000$ simulations per attribute sample.
    }\label{fig:m_sweep}
\end{figure}

\autoref{fig:delta_sweep} shows $\overline{T}$ and $\overline{F}$ across $40$ equidistant $\delta\in[10^{-1.1},1]$. The estimates remain separated throughout the interval, with $\overline{T}>\overline{F}$ for all considered values of $\delta$, and both approach $1$ as $\delta$ tends to $1$. 
The monotonic increase in $\overline{F}$ is due to the role of $\delta$ in determining the size of the bad regions. Recall that the Gaussian noise level is calibrated by minimising \eqref{eq:lagrange}, thereby enforcing $P_x(S^\delta_{x,x'}) \approx \delta$. Consequently, larger values of $\delta$ correspond to larger bad regions, which in turn increase both $P_x(S^\delta_x)$ and the mean theoretical upper limit $\overline{T}$. While $\overline{T}$ increases steadily across the entire range, most of the increase in $\overline{F}$ occurs in the interval $\delta\in[0.95,1]$. Specifically, $\overline{F}$ grows from $18\%$ to $100\%$ over $\delta\in[0.95,1]$, compared with an increase from $0.007\%$ to $18\%$  over $\delta\in[10^{-1.1},0.95]$. In practice, $\delta$ is typically chosen considerably smaller than $10^{-1.1}$ \citep{dwork2006our}. Consequently, practical applications of posterior-to-prior ratio bounds are unlikely to operate in the regime of the rapid failure probability increase observed for $\delta$ close to one. We have not obtained estimates for $\delta<10^{-1.1}$, as the computational requirements of the corresponding rare-event simulations would be prohibitive.
The standard deviations of $\overline{T}$ and $\overline{F}$ are within $[1.11\cdot10^{-19},4.33\cdot10^{-3}]$ and $[8.44\cdot10^{-12},4.25\cdot10^{-10}]$, respectively, and are therefore negligible on the scale of \autoref{fig:delta_sweep}. 

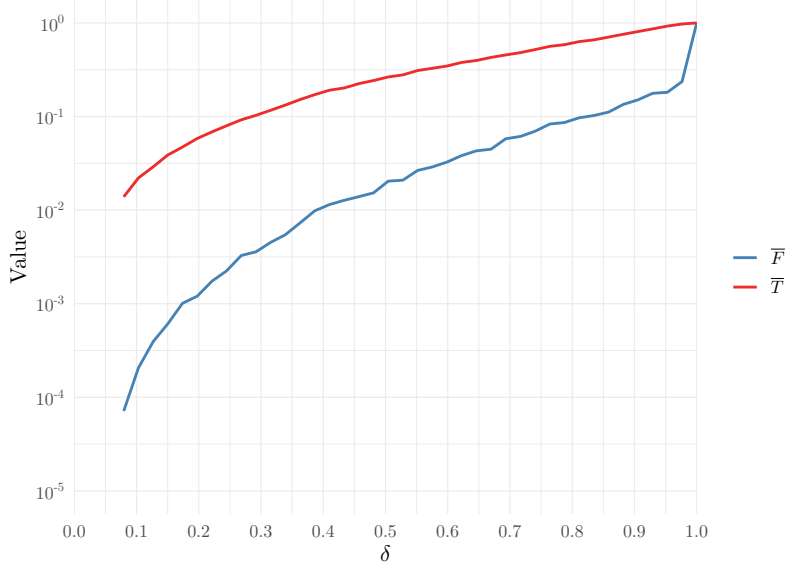
\begin{figure}[H]
    \centering
    \resizebox{0.75\textwidth}{!}{
        \input{figures/delta_sweep_plot_v3}
    }
    \caption{Empirical estimates and standard deviations of the failure probability $\overline{F}$ (blue) and the corresponding mean theoretical upper limits $\overline{T}$ (red) for $40$ equidistant values of $\delta$ in the interval $[10^{-1.1},1]$. Results are shown for $n=100$, $m=90$, and $\varepsilon=1$. Estimates are based on $N_1=1,000$ attribute samples and $N_2=1,000,000$ simulations per attribute sample.
    }\label{fig:delta_sweep}
\end{figure}

In summary, Monte Carlo simulations across a range of values of $\varepsilon$, $m$, and $\delta$ indicate that the theoretical upper limit $P_{x}(S^{\delta}_{x})$ is highly conservative compared to the true failure probability. In realistic settings, the privacy budget $\varepsilon$ is likely to be moderate, the population size $n$ substantially larger than the sample size $m$, and $\delta$ negligibly small. Consquently, the true failure probability may be several orders of magnitude below its theoretical upper limit. The principal implication is that if privacy parameters can be chosen such the posterior-to-prior ratio is sufficiently constrained by \eqref{eq:pDPBounds}, then $P_{x}(S^{\delta}_{x})$ provides a useful conservative limit on the probability that the privacy guarantee is violated.

%% file: figures/epsilon_sweep_plot_v3.tex
\begin{tikzpicture}[x=1pt,y=1pt]
\definecolor{fillColor}{RGB}{255,255,255}
\path[use as bounding box,fill=fillColor,fill opacity=0.00] (0,0) rectangle (433.62,289.08);
\begin{scope}
\path[clip] (  0.00,  0.00) rectangle (433.62,289.08);
\definecolor{fillColor}{RGB}{255,255,255}

\path[fill=fillColor] (  0.00,  0.00) rectangle (433.62,289.08);
\end{scope}
\begin{scope}
\path[clip] ( 36.84, 30.69) rectangle (342.64,283.58);
\definecolor{drawColor}{gray}{0.92}

\path[draw=drawColor,line width= 0.3pt,line join=round] ( 36.84, 40.12) --
	(342.64, 40.12);

\path[draw=drawColor,line width= 0.3pt,line join=round] ( 36.84,124.25) --
	(342.64,124.25);

\path[draw=drawColor,line width= 0.3pt,line join=round] ( 36.84,208.38) --
	(342.64,208.38);

\path[draw=drawColor,line width= 0.3pt,line join=round] ( 66.19, 30.69) --
	( 66.19,283.58);

\path[draw=drawColor,line width= 0.3pt,line join=round] ( 97.08, 30.69) --
	( 97.08,283.58);

\path[draw=drawColor,line width= 0.3pt,line join=round] (127.97, 30.69) --
	(127.97,283.58);

\path[draw=drawColor,line width= 0.3pt,line join=round] (158.85, 30.69) --
	(158.85,283.58);

\path[draw=drawColor,line width= 0.3pt,line join=round] (189.74, 30.69) --
	(189.74,283.58);

\path[draw=drawColor,line width= 0.3pt,line join=round] (220.63, 30.69) --
	(220.63,283.58);

\path[draw=drawColor,line width= 0.3pt,line join=round] (251.52, 30.69) --
	(251.52,283.58);

\path[draw=drawColor,line width= 0.3pt,line join=round] (282.41, 30.69) --
	(282.41,283.58);

\path[draw=drawColor,line width= 0.3pt,line join=round] (313.30, 30.69) --
	(313.30,283.58);

\path[draw=drawColor,line width= 0.6pt,line join=round] ( 36.84, 82.19) --
	(342.64, 82.19);

\path[draw=drawColor,line width= 0.6pt,line join=round] ( 36.84,166.31) --
	(342.64,166.31);

\path[draw=drawColor,line width= 0.6pt,line join=round] ( 36.84,250.44) --
	(342.64,250.44);

\path[draw=drawColor,line width= 0.6pt,line join=round] ( 50.74, 30.69) --
	( 50.74,283.58);

\path[draw=drawColor,line width= 0.6pt,line join=round] ( 81.63, 30.69) --
	( 81.63,283.58);

\path[draw=drawColor,line width= 0.6pt,line join=round] (112.52, 30.69) --
	(112.52,283.58);

\path[draw=drawColor,line width= 0.6pt,line join=round] (143.41, 30.69) --
	(143.41,283.58);

\path[draw=drawColor,line width= 0.6pt,line join=round] (174.30, 30.69) --
	(174.30,283.58);

\path[draw=drawColor,line width= 0.6pt,line join=round] (205.19, 30.69) --
	(205.19,283.58);

\path[draw=drawColor,line width= 0.6pt,line join=round] (236.08, 30.69) --
	(236.08,283.58);

\path[draw=drawColor,line width= 0.6pt,line join=round] (266.96, 30.69) --
	(266.96,283.58);

\path[draw=drawColor,line width= 0.6pt,line join=round] (297.85, 30.69) --
	(297.85,283.58);

\path[draw=drawColor,line width= 0.6pt,line join=round] (328.74, 30.69) --
	(328.74,283.58);
\definecolor{fillColor}{RGB}{70,130,180}

\path[fill=fillColor,fill opacity=0.20] ( 50.74,101.57) --
	( 57.87, 95.22) --
	( 65.00,112.52) --
	( 72.13, 98.88) --
	( 79.26, 95.35) --
	( 86.38, 97.68) --
	( 93.51,102.93) --
	(100.64, 98.76) --
	(107.77, 99.04) --
	(114.90,101.71) --
	(122.03, 94.14) --
	(129.15, 94.00) --
	(136.28, 98.27) --
	(143.41, 86.21) --
	(150.54, 93.96) --
	(157.67, 85.30) --
	(164.79, 94.07) --
	(171.92, 96.29) --
	(179.05, 81.08) --
	(186.18, 87.81) --
	(193.31, 84.53) --
	(200.43, 83.18) --
	(207.56, 80.71) --
	(214.69, 74.20) --
	(221.82, 72.71) --
	(228.95, 80.53) --
	(236.08, 71.55) --
	(243.20, 82.07) --
	(250.33, 77.10) --
	(257.46, 66.27) --
	(264.59, 74.96) --
	(271.72, 58.16) --
	(278.84, 55.61) --
	(285.97, 57.06) --
	(293.10, 56.59) --
	(300.23, 64.72) --
	(307.36, 58.49) --
	(314.48, 59.37) --
	(321.61, 61.62) --
	(328.74, 42.96) --
	(328.74, 42.18) --
	(321.61, 61.02) --
	(314.48, 58.74) --
	(307.36, 57.86) --
	(300.23, 64.14) --
	(293.10, 55.94) --
	(285.97, 56.42) --
	(278.84, 54.95) --
	(271.72, 57.52) --
	(264.59, 74.46) --
	(257.46, 65.71) --
	(250.33, 76.61) --
	(243.20, 81.62) --
	(236.08, 71.02) --
	(228.95, 80.06) --
	(221.82, 72.19) --
	(214.69, 73.69) --
	(207.56, 80.25) --
	(200.43, 82.73) --
	(193.31, 84.09) --
	(186.18, 87.38) --
	(179.05, 80.62) --
	(171.92, 95.91) --
	(164.79, 93.69) --
	(157.67, 84.87) --
	(150.54, 93.58) --
	(143.41, 85.78) --
	(136.28, 97.90) --
	(129.15, 93.61) --
	(122.03, 93.75) --
	(114.90,101.36) --
	(107.77, 98.68) --
	(100.64, 98.40) --
	( 93.51,102.59) --
	( 86.38, 97.31) --
	( 79.26, 94.97) --
	( 72.13, 98.52) --
	( 65.00,112.22) --
	( 57.87, 94.84) --
	( 50.74,101.23) --
	cycle;

\path[] ( 50.74,101.57) --
	( 57.87, 95.22) --
	( 65.00,112.52) --
	( 72.13, 98.88) --
	( 79.26, 95.35) --
	( 86.38, 97.68) --
	( 93.51,102.93) --
	(100.64, 98.76) --
	(107.77, 99.04) --
	(114.90,101.71) --
	(122.03, 94.14) --
	(129.15, 94.00) --
	(136.28, 98.27) --
	(143.41, 86.21) --
	(150.54, 93.96) --
	(157.67, 85.30) --
	(164.79, 94.07) --
	(171.92, 96.29) --
	(179.05, 81.08) --
	(186.18, 87.81) --
	(193.31, 84.53) --
	(200.43, 83.18) --
	(207.56, 80.71) --
	(214.69, 74.20) --
	(221.82, 72.71) --
	(228.95, 80.53) --
	(236.08, 71.55) --
	(243.20, 82.07) --
	(250.33, 77.10) --
	(257.46, 66.27) --
	(264.59, 74.96) --
	(271.72, 58.16) --
	(278.84, 55.61) --
	(285.97, 57.06) --
	(293.10, 56.59) --
	(300.23, 64.72) --
	(307.36, 58.49) --
	(314.48, 59.37) --
	(321.61, 61.62) --
	(328.74, 42.96);

\path[] (328.74, 42.18) --
	(321.61, 61.02) --
	(314.48, 58.74) --
	(307.36, 57.86) --
	(300.23, 64.14) --
	(293.10, 55.94) --
	(285.97, 56.42) --
	(278.84, 54.95) --
	(271.72, 57.52) --
	(264.59, 74.46) --
	(257.46, 65.71) --
	(250.33, 76.61) --
	(243.20, 81.62) --
	(236.08, 71.02) --
	(228.95, 80.06) --
	(221.82, 72.19) --
	(214.69, 73.69) --
	(207.56, 80.25) --
	(200.43, 82.73) --
	(193.31, 84.09) --
	(186.18, 87.38) --
	(179.05, 80.62) --
	(171.92, 95.91) --
	(164.79, 93.69) --
	(157.67, 84.87) --
	(150.54, 93.58) --
	(143.41, 85.78) --
	(136.28, 97.90) --
	(129.15, 93.61) --
	(122.03, 93.75) --
	(114.90,101.36) --
	(107.77, 98.68) --
	(100.64, 98.40) --
	( 93.51,102.59) --
	( 86.38, 97.31) --
	( 79.26, 94.97) --
	( 72.13, 98.52) --
	( 65.00,112.22) --
	( 57.87, 94.84) --
	( 50.74,101.23);
\definecolor{fillColor}{RGB}{238,44,44}

\path[fill=fillColor,fill opacity=0.15] ( 50.74,268.58) --
	( 57.87,268.84) --
	( 65.00,270.93) --
	( 72.13,272.08) --
	( 79.26,265.34) --
	( 86.38,268.55) --
	( 93.51,271.80) --
	(100.64,265.26) --
	(107.77,267.52) --
	(114.90,268.36) --
	(122.03,267.87) --
	(129.15,268.93) --
	(136.28,268.14) --
	(143.41,269.29) --
	(150.54,269.71) --
	(157.67,265.19) --
	(164.79,269.15) --
	(171.92,270.53) --
	(179.05,267.98) --
	(186.18,264.96) --
	(193.31,266.69) --
	(200.43,267.03) --
	(207.56,266.27) --
	(214.69,265.16) --
	(221.82,264.58) --
	(228.95,267.04) --
	(236.08,266.66) --
	(243.20,266.63) --
	(250.33,264.04) --
	(257.46,262.34) --
	(264.59,264.86) --
	(271.72,263.79) --
	(278.84,260.83) --
	(285.97,264.88) --
	(293.10,263.41) --
	(300.23,263.94) --
	(307.36,263.12) --
	(314.48,262.49) --
	(321.61,261.91) --
	(328.74,259.64) --
	(328.74,256.59) --
	(321.61,258.92) --
	(314.48,259.44) --
	(307.36,260.09) --
	(300.23,260.96) --
	(293.10,260.48) --
	(285.97,261.93) --
	(278.84,257.65) --
	(271.72,260.83) --
	(264.59,261.99) --
	(257.46,259.35) --
	(250.33,261.09) --
	(243.20,263.74) --
	(236.08,263.74) --
	(228.95,264.25) --
	(221.82,261.68) --
	(214.69,262.36) --
	(207.56,263.42) --
	(200.43,264.20) --
	(193.31,263.87) --
	(186.18,262.03) --
	(179.05,265.19) --
	(171.92,267.80) --
	(164.79,266.47) --
	(157.67,262.30) --
	(150.54,267.01) --
	(143.41,266.61) --
	(136.28,265.35) --
	(129.15,266.19) --
	(122.03,265.09) --
	(114.90,265.64) --
	(107.77,264.74) --
	(100.64,262.47) --
	( 93.51,269.24) --
	( 86.38,265.82) --
	( 79.26,262.46) --
	( 72.13,269.53) --
	( 65.00,268.35) --
	( 57.87,266.13) --
	( 50.74,265.82) --
	cycle;

\path[] ( 50.74,268.58) --
	( 57.87,268.84) --
	( 65.00,270.93) --
	( 72.13,272.08) --
	( 79.26,265.34) --
	( 86.38,268.55) --
	( 93.51,271.80) --
	(100.64,265.26) --
	(107.77,267.52) --
	(114.90,268.36) --
	(122.03,267.87) --
	(129.15,268.93) --
	(136.28,268.14) --
	(143.41,269.29) --
	(150.54,269.71) --
	(157.67,265.19) --
	(164.79,269.15) --
	(171.92,270.53) --
	(179.05,267.98) --
	(186.18,264.96) --
	(193.31,266.69) --
	(200.43,267.03) --
	(207.56,266.27) --
	(214.69,265.16) --
	(221.82,264.58) --
	(228.95,267.04) --
	(236.08,266.66) --
	(243.20,266.63) --
	(250.33,264.04) --
	(257.46,262.34) --
	(264.59,264.86) --
	(271.72,263.79) --
	(278.84,260.83) --
	(285.97,264.88) --
	(293.10,263.41) --
	(300.23,263.94) --
	(307.36,263.12) --
	(314.48,262.49) --
	(321.61,261.91) --
	(328.74,259.64);

\path[] (328.74,256.59) --
	(321.61,258.92) --
	(314.48,259.44) --
	(307.36,260.09) --
	(300.23,260.96) --
	(293.10,260.48) --
	(285.97,261.93) --
	(278.84,257.65) --
	(271.72,260.83) --
	(264.59,261.99) --
	(257.46,259.35) --
	(250.33,261.09) --
	(243.20,263.74) --
	(236.08,263.74) --
	(228.95,264.25) --
	(221.82,261.68) --
	(214.69,262.36) --
	(207.56,263.42) --
	(200.43,264.20) --
	(193.31,263.87) --
	(186.18,262.03) --
	(179.05,265.19) --
	(171.92,267.80) --
	(164.79,266.47) --
	(157.67,262.30) --
	(150.54,267.01) --
	(143.41,266.61) --
	(136.28,265.35) --
	(129.15,266.19) --
	(122.03,265.09) --
	(114.90,265.64) --
	(107.77,264.74) --
	(100.64,262.47) --
	( 93.51,269.24) --
	( 86.38,265.82) --
	( 79.26,262.46) --
	( 72.13,269.53) --
	( 65.00,268.35) --
	( 57.87,266.13) --
	( 50.74,265.82);
\definecolor{drawColor}{RGB}{70,130,180}

\path[draw=drawColor,line width= 1.4pt,line join=round] ( 50.74,101.40) --
	( 57.87, 95.03) --
	( 65.00,112.37) --
	( 72.13, 98.70) --
	( 79.26, 95.16) --
	( 86.38, 97.50) --
	( 93.51,102.76) --
	(100.64, 98.58) --
	(107.77, 98.86) --
	(114.90,101.54) --
	(122.03, 93.95) --
	(129.15, 93.80) --
	(136.28, 98.08) --
	(143.41, 85.99) --
	(150.54, 93.77) --
	(157.67, 85.09) --
	(164.79, 93.88) --
	(171.92, 96.10) --
	(179.05, 80.85) --
	(186.18, 87.60) --
	(193.31, 84.31) --
	(200.43, 82.95) --
	(207.56, 80.48) --
	(214.69, 73.95) --
	(221.82, 72.45) --
	(228.95, 80.30) --
	(236.08, 71.28) --
	(243.20, 81.85) --
	(250.33, 76.86) --
	(257.46, 65.99) --
	(264.59, 74.71) --
	(271.72, 57.84) --
	(278.84, 55.28) --
	(285.97, 56.74) --
	(293.10, 56.27) --
	(300.23, 64.43) --
	(307.36, 58.17) --
	(314.48, 59.06) --
	(321.61, 61.32) --
	(328.74, 42.57);
\definecolor{drawColor}{RGB}{238,44,44}

\path[draw=drawColor,line width= 1.4pt,line join=round] ( 50.74,267.22) --
	( 57.87,267.51) --
	( 65.00,269.66) --
	( 72.13,270.83) --
	( 79.26,263.93) --
	( 86.38,267.21) --
	( 93.51,270.54) --
	(100.64,263.89) --
	(107.77,266.16) --
	(114.90,267.03) --
	(122.03,266.50) --
	(129.15,267.59) --
	(136.28,266.77) --
	(143.41,267.97) --
	(150.54,268.39) --
	(157.67,263.77) --
	(164.79,267.83) --
	(171.92,269.19) --
	(179.05,266.61) --
	(186.18,263.52) --
	(193.31,265.31) --
	(200.43,265.64) --
	(207.56,264.87) --
	(214.69,263.79) --
	(221.82,263.16) --
	(228.95,265.67) --
	(236.08,265.23) --
	(243.20,265.22) --
	(250.33,262.59) --
	(257.46,260.88) --
	(264.59,263.45) --
	(271.72,262.34) --
	(278.84,259.27) --
	(285.97,263.43) --
	(293.10,261.97) --
	(300.23,262.48) --
	(307.36,261.64) --
	(314.48,261.00) --
	(321.61,260.45) --
	(328.74,258.14);
\end{scope}
\begin{scope}
\path[clip] (  0.00,  0.00) rectangle (433.62,289.08);
\definecolor{drawColor}{gray}{0.30}

\node[text=drawColor,anchor=base west,inner sep=0pt, outer sep=0pt, scale=  0.88] at ( 17.96, 78.41) {10};

\node[text=drawColor,anchor=base west,inner sep=0pt, outer sep=0pt, scale=  0.62] at ( 26.76, 82.01) {-};

\node[text=drawColor,anchor=base west,inner sep=0pt, outer sep=0pt, scale=  0.62] at ( 28.82, 82.01) {4};

\node[text=drawColor,anchor=base west,inner sep=0pt, outer sep=0pt, scale=  0.88] at ( 17.96,162.54) {10};

\node[text=drawColor,anchor=base west,inner sep=0pt, outer sep=0pt, scale=  0.62] at ( 26.76,166.14) {-};

\node[text=drawColor,anchor=base west,inner sep=0pt, outer sep=0pt, scale=  0.62] at ( 28.82,166.14) {3};

\node[text=drawColor,anchor=base west,inner sep=0pt, outer sep=0pt, scale=  0.88] at ( 17.96,246.67) {10};

\node[text=drawColor,anchor=base west,inner sep=0pt, outer sep=0pt, scale=  0.62] at ( 26.76,250.27) {-};

\node[text=drawColor,anchor=base west,inner sep=0pt, outer sep=0pt, scale=  0.62] at ( 28.82,250.27) {2};
\end{scope}
\begin{scope}
\path[clip] (  0.00,  0.00) rectangle (433.62,289.08);
\definecolor{drawColor}{gray}{0.30}

\node[text=drawColor,anchor=base,inner sep=0pt, outer sep=0pt, scale=  0.88] at ( 50.74, 19.68) {0.2};

\node[text=drawColor,anchor=base,inner sep=0pt, outer sep=0pt, scale=  0.88] at ( 81.63, 19.68) {0.4};

\node[text=drawColor,anchor=base,inner sep=0pt, outer sep=0pt, scale=  0.88] at (112.52, 19.68) {0.6};

\node[text=drawColor,anchor=base,inner sep=0pt, outer sep=0pt, scale=  0.88] at (143.41, 19.68) {0.8};

\node[text=drawColor,anchor=base,inner sep=0pt, outer sep=0pt, scale=  0.88] at (174.30, 19.68) {1.0};

\node[text=drawColor,anchor=base,inner sep=0pt, outer sep=0pt, scale=  0.88] at (205.19, 19.68) {1.2};

\node[text=drawColor,anchor=base,inner sep=0pt, outer sep=0pt, scale=  0.88] at (236.08, 19.68) {1.4};

\node[text=drawColor,anchor=base,inner sep=0pt, outer sep=0pt, scale=  0.88] at (266.96, 19.68) {1.6};

\node[text=drawColor,anchor=base,inner sep=0pt, outer sep=0pt, scale=  0.88] at (297.85, 19.68) {1.8};

\node[text=drawColor,anchor=base,inner sep=0pt, outer sep=0pt, scale=  0.88] at (328.74, 19.68) {2.0};
\end{scope}
\begin{scope}
\path[clip] (  0.00,  0.00) rectangle (433.62,289.08);
\definecolor{drawColor}{RGB}{0,0,0}

\node[text=drawColor,anchor=base,inner sep=0pt, outer sep=0pt, scale=  1.10] at (189.74,  7.64) {$\varepsilon$};
\end{scope}
\begin{scope}
\path[clip] (  0.00,  0.00) rectangle (433.62,289.08);
\definecolor{drawColor}{RGB}{0,0,0}

\node[text=drawColor,rotate= 90.00,anchor=base,inner sep=0pt, outer sep=0pt, scale=  1.10] at ( 13.08,157.13) {Value};
\end{scope}
\begin{scope}
\path[clip] (  0.00,  0.00) rectangle (433.62,289.08);
\definecolor{drawColor}{RGB}{70,130,180}

\path[draw=drawColor,line width= 1.4pt,line join=round] (360.59,156.75) -- (372.15,156.75);
\end{scope}
\begin{scope}
\path[clip] (  0.00,  0.00) rectangle (433.62,289.08);
\definecolor{drawColor}{RGB}{238,44,44}

\path[draw=drawColor,line width= 1.4pt,line join=round] (360.59,142.30) -- (372.15,142.30);
\end{scope}
\begin{scope}
\path[clip] (  0.00,  0.00) rectangle (433.62,289.08);
\definecolor{drawColor}{RGB}{0,0,0}

\node[text=drawColor,anchor=base west,inner sep=0pt, outer sep=0pt, scale=  0.88] at (379.09,153.72) {$\overline{F}$};
\end{scope}
\begin{scope}
\path[clip] (  0.00,  0.00) rectangle (433.62,289.08);
\definecolor{drawColor}{RGB}{0,0,0}

\node[text=drawColor,anchor=base west,inner sep=0pt, outer sep=0pt, scale=  0.88] at (379.09,139.27) {$\overline{T}$};
\end{scope}
\end{tikzpicture}

%% file: figures/m_sweep_plot_v3.tex
\begin{tikzpicture}[x=1pt,y=1pt]
\definecolor{fillColor}{RGB}{255,255,255}
\path[use as bounding box,fill=fillColor,fill opacity=0.00] (0,0) rectangle (433.62,289.08);
\begin{scope}
\path[clip] (  0.00,  0.00) rectangle (433.62,289.08);
\definecolor{fillColor}{RGB}{255,255,255}

\path[fill=fillColor] (  0.00,  0.00) rectangle (433.62,289.08);
\end{scope}
\begin{scope}
\path[clip] ( 36.84, 30.69) rectangle (342.64,283.58);
\definecolor{drawColor}{gray}{0.92}

\path[draw=drawColor,line width= 0.3pt,line join=round] ( 36.84, 70.92) --
	(342.64, 70.92);

\path[draw=drawColor,line width= 0.3pt,line join=round] ( 36.84,128.39) --
	(342.64,128.39);

\path[draw=drawColor,line width= 0.3pt,line join=round] ( 36.84,185.87) --
	(342.64,185.87);

\path[draw=drawColor,line width= 0.3pt,line join=round] ( 36.84,243.35) --
	(342.64,243.35);

\path[draw=drawColor,line width= 0.3pt,line join=round] ( 66.19, 30.69) --
	( 66.19,283.58);

\path[draw=drawColor,line width= 0.3pt,line join=round] ( 97.08, 30.69) --
	( 97.08,283.58);

\path[draw=drawColor,line width= 0.3pt,line join=round] (127.97, 30.69) --
	(127.97,283.58);

\path[draw=drawColor,line width= 0.3pt,line join=round] (158.85, 30.69) --
	(158.85,283.58);

\path[draw=drawColor,line width= 0.3pt,line join=round] (189.74, 30.69) --
	(189.74,283.58);

\path[draw=drawColor,line width= 0.3pt,line join=round] (220.63, 30.69) --
	(220.63,283.58);

\path[draw=drawColor,line width= 0.3pt,line join=round] (251.52, 30.69) --
	(251.52,283.58);

\path[draw=drawColor,line width= 0.3pt,line join=round] (282.41, 30.69) --
	(282.41,283.58);

\path[draw=drawColor,line width= 0.3pt,line join=round] (313.30, 30.69) --
	(313.30,283.58);

\path[draw=drawColor,line width= 0.6pt,line join=round] ( 36.84, 42.18) --
	(342.64, 42.18);

\path[draw=drawColor,line width= 0.6pt,line join=round] ( 36.84, 99.66) --
	(342.64, 99.66);

\path[draw=drawColor,line width= 0.6pt,line join=round] ( 36.84,157.13) --
	(342.64,157.13);

\path[draw=drawColor,line width= 0.6pt,line join=round] ( 36.84,214.61) --
	(342.64,214.61);

\path[draw=drawColor,line width= 0.6pt,line join=round] ( 36.84,272.08) --
	(342.64,272.08);

\path[draw=drawColor,line width= 0.6pt,line join=round] ( 50.74, 30.69) --
	( 50.74,283.58);

\path[draw=drawColor,line width= 0.6pt,line join=round] ( 81.63, 30.69) --
	( 81.63,283.58);

\path[draw=drawColor,line width= 0.6pt,line join=round] (112.52, 30.69) --
	(112.52,283.58);

\path[draw=drawColor,line width= 0.6pt,line join=round] (143.41, 30.69) --
	(143.41,283.58);

\path[draw=drawColor,line width= 0.6pt,line join=round] (174.30, 30.69) --
	(174.30,283.58);

\path[draw=drawColor,line width= 0.6pt,line join=round] (205.19, 30.69) --
	(205.19,283.58);

\path[draw=drawColor,line width= 0.6pt,line join=round] (236.08, 30.69) --
	(236.08,283.58);

\path[draw=drawColor,line width= 0.6pt,line join=round] (266.96, 30.69) --
	(266.96,283.58);

\path[draw=drawColor,line width= 0.6pt,line join=round] (297.85, 30.69) --
	(297.85,283.58);

\path[draw=drawColor,line width= 0.6pt,line join=round] (328.74, 30.69) --
	(328.74,283.58);
\definecolor{fillColor}{RGB}{70,130,180}

\path[fill=fillColor,fill opacity=0.20] ( 50.74, 76.65) --
	( 56.92, 79.16) --
	( 66.19, 77.34) --
	( 72.37, 80.79) --
	( 78.54, 76.19) --
	( 84.72, 77.40) --
	( 93.99, 75.45) --
	(100.17, 76.95) --
	(106.34, 78.80) --
	(115.61, 74.69) --
	(121.79, 77.42) --
	(127.97, 75.74) --
	(134.14, 75.14) --
	(143.41, 78.02) --
	(149.59, 78.61) --
	(155.77, 81.81) --
	(165.03, 79.11) --
	(171.21, 75.38) --
	(177.39, 82.43) --
	(183.56, 83.41) --
	(192.83, 78.43) --
	(199.01, 79.37) --
	(205.19, 83.82) --
	(211.36, 83.77) --
	(220.63, 82.07) --
	(226.81, 83.15) --
	(232.99, 84.85) --
	(242.25, 86.91) --
	(248.43, 87.29) --
	(254.61, 84.42) --
	(260.79, 89.21) --
	(270.05, 89.64) --
	(276.23, 86.58) --
	(282.41, 95.52) --
	(291.67, 96.49) --
	(297.85,111.24) --
	(304.03,111.65) --
	(310.21,119.50) --
	(319.47,142.17) --
	(325.65,186.08) --
	(325.65,186.03) --
	(319.47,142.04) --
	(310.21,119.29) --
	(304.03,111.41) --
	(297.85,111.00) --
	(291.67, 96.16) --
	(282.41, 95.18) --
	(276.23, 86.18) --
	(270.05, 89.26) --
	(260.79, 88.83) --
	(254.61, 84.00) --
	(248.43, 86.89) --
	(242.25, 86.51) --
	(232.99, 84.43) --
	(226.81, 82.72) --
	(220.63, 81.62) --
	(211.36, 83.35) --
	(205.19, 83.39) --
	(199.01, 78.90) --
	(192.83, 77.95) --
	(183.56, 82.98) --
	(177.39, 81.99) --
	(171.21, 74.88) --
	(165.03, 78.64) --
	(155.77, 81.36) --
	(149.59, 78.13) --
	(143.41, 77.54) --
	(134.14, 74.63) --
	(127.97, 75.23) --
	(121.79, 76.94) --
	(115.61, 74.17) --
	(106.34, 78.33) --
	(100.17, 76.46) --
	( 93.99, 74.94) --
	( 84.72, 76.91) --
	( 78.54, 75.70) --
	( 72.37, 80.34) --
	( 66.19, 76.85) --
	( 56.92, 78.69) --
	( 50.74, 76.15) --
	cycle;

\path[] ( 50.74, 76.65) --
	( 56.92, 79.16) --
	( 66.19, 77.34) --
	( 72.37, 80.79) --
	( 78.54, 76.19) --
	( 84.72, 77.40) --
	( 93.99, 75.45) --
	(100.17, 76.95) --
	(106.34, 78.80) --
	(115.61, 74.69) --
	(121.79, 77.42) --
	(127.97, 75.74) --
	(134.14, 75.14) --
	(143.41, 78.02) --
	(149.59, 78.61) --
	(155.77, 81.81) --
	(165.03, 79.11) --
	(171.21, 75.38) --
	(177.39, 82.43) --
	(183.56, 83.41) --
	(192.83, 78.43) --
	(199.01, 79.37) --
	(205.19, 83.82) --
	(211.36, 83.77) --
	(220.63, 82.07) --
	(226.81, 83.15) --
	(232.99, 84.85) --
	(242.25, 86.91) --
	(248.43, 87.29) --
	(254.61, 84.42) --
	(260.79, 89.21) --
	(270.05, 89.64) --
	(276.23, 86.58) --
	(282.41, 95.52) --
	(291.67, 96.49) --
	(297.85,111.24) --
	(304.03,111.65) --
	(310.21,119.50) --
	(319.47,142.17) --
	(325.65,186.08);

\path[] (325.65,186.03) --
	(319.47,142.04) --
	(310.21,119.29) --
	(304.03,111.41) --
	(297.85,111.00) --
	(291.67, 96.16) --
	(282.41, 95.18) --
	(276.23, 86.18) --
	(270.05, 89.26) --
	(260.79, 88.83) --
	(254.61, 84.00) --
	(248.43, 86.89) --
	(242.25, 86.51) --
	(232.99, 84.43) --
	(226.81, 82.72) --
	(220.63, 81.62) --
	(211.36, 83.35) --
	(205.19, 83.39) --
	(199.01, 78.90) --
	(192.83, 77.95) --
	(183.56, 82.98) --
	(177.39, 81.99) --
	(171.21, 74.88) --
	(165.03, 78.64) --
	(155.77, 81.36) --
	(149.59, 78.13) --
	(143.41, 77.54) --
	(134.14, 74.63) --
	(127.97, 75.23) --
	(121.79, 76.94) --
	(115.61, 74.17) --
	(106.34, 78.33) --
	(100.17, 76.46) --
	( 93.99, 74.94) --
	( 84.72, 76.91) --
	( 78.54, 75.70) --
	( 72.37, 80.34) --
	( 66.19, 76.85) --
	( 56.92, 78.69) --
	( 50.74, 76.15);
\definecolor{fillColor}{RGB}{238,44,44}

\path[fill=fillColor,fill opacity=0.15] ( 50.74,237.00) --
	( 56.92,237.80) --
	( 66.19,236.98) --
	( 72.37,238.00) --
	( 78.54,236.55) --
	( 84.72,237.15) --
	( 93.99,235.66) --
	(100.17,236.93) --
	(106.34,237.02) --
	(115.61,235.43) --
	(121.79,235.95) --
	(127.97,235.49) --
	(134.14,236.13) --
	(143.41,235.64) --
	(149.59,235.65) --
	(155.77,236.41) --
	(165.03,235.87) --
	(171.21,234.58) --
	(177.39,235.53) --
	(183.56,235.33) --
	(192.83,235.10) --
	(199.01,235.51) --
	(205.19,236.53) --
	(211.36,235.40) --
	(220.63,235.13) --
	(226.81,234.79) --
	(232.99,234.90) --
	(242.25,233.53) --
	(248.43,233.33) --
	(254.61,233.02) --
	(260.79,231.73) --
	(270.05,230.82) --
	(276.23,229.50) --
	(282.41,229.53) --
	(291.67,228.70) --
	(297.85,225.27) --
	(304.03,223.26) --
	(310.21,220.24) --
	(319.47,207.09) --
	(325.65,188.27) --
	(325.65,183.61) --
	(319.47,204.13) --
	(310.21,218.00) --
	(304.03,221.22) --
	(297.85,223.32) --
	(291.67,226.92) --
	(282.41,227.77) --
	(276.23,227.73) --
	(270.05,229.08) --
	(260.79,230.02) --
	(254.61,231.41) --
	(248.43,231.74) --
	(242.25,231.92) --
	(232.99,233.38) --
	(226.81,233.24) --
	(220.63,233.58) --
	(211.36,233.82) --
	(205.19,235.02) --
	(199.01,233.99) --
	(192.83,233.51) --
	(183.56,233.74) --
	(177.39,233.97) --
	(171.21,233.01) --
	(165.03,234.32) --
	(155.77,234.87) --
	(149.59,234.10) --
	(143.41,234.11) --
	(134.14,234.63) --
	(127.97,233.95) --
	(121.79,234.41) --
	(115.61,233.88) --
	(106.34,235.51) --
	(100.17,235.42) --
	( 93.99,234.10) --
	( 84.72,235.64) --
	( 78.54,235.00) --
	( 72.37,236.50) --
	( 66.19,235.47) --
	( 56.92,236.32) --
	( 50.74,235.53) --
	cycle;

\path[] ( 50.74,237.00) --
	( 56.92,237.80) --
	( 66.19,236.98) --
	( 72.37,238.00) --
	( 78.54,236.55) --
	( 84.72,237.15) --
	( 93.99,235.66) --
	(100.17,236.93) --
	(106.34,237.02) --
	(115.61,235.43) --
	(121.79,235.95) --
	(127.97,235.49) --
	(134.14,236.13) --
	(143.41,235.64) --
	(149.59,235.65) --
	(155.77,236.41) --
	(165.03,235.87) --
	(171.21,234.58) --
	(177.39,235.53) --
	(183.56,235.33) --
	(192.83,235.10) --
	(199.01,235.51) --
	(205.19,236.53) --
	(211.36,235.40) --
	(220.63,235.13) --
	(226.81,234.79) --
	(232.99,234.90) --
	(242.25,233.53) --
	(248.43,233.33) --
	(254.61,233.02) --
	(260.79,231.73) --
	(270.05,230.82) --
	(276.23,229.50) --
	(282.41,229.53) --
	(291.67,228.70) --
	(297.85,225.27) --
	(304.03,223.26) --
	(310.21,220.24) --
	(319.47,207.09) --
	(325.65,188.27);

\path[] (325.65,183.61) --
	(319.47,204.13) --
	(310.21,218.00) --
	(304.03,221.22) --
	(297.85,223.32) --
	(291.67,226.92) --
	(282.41,227.77) --
	(276.23,227.73) --
	(270.05,229.08) --
	(260.79,230.02) --
	(254.61,231.41) --
	(248.43,231.74) --
	(242.25,231.92) --
	(232.99,233.38) --
	(226.81,233.24) --
	(220.63,233.58) --
	(211.36,233.82) --
	(205.19,235.02) --
	(199.01,233.99) --
	(192.83,233.51) --
	(183.56,233.74) --
	(177.39,233.97) --
	(171.21,233.01) --
	(165.03,234.32) --
	(155.77,234.87) --
	(149.59,234.10) --
	(143.41,234.11) --
	(134.14,234.63) --
	(127.97,233.95) --
	(121.79,234.41) --
	(115.61,233.88) --
	(106.34,235.51) --
	(100.17,235.42) --
	( 93.99,234.10) --
	( 84.72,235.64) --
	( 78.54,235.00) --
	( 72.37,236.50) --
	( 66.19,235.47) --
	( 56.92,236.32) --
	( 50.74,235.53);
\definecolor{drawColor}{RGB}{70,130,180}

\path[draw=drawColor,line width= 1.4pt,line join=round] ( 50.74, 76.40) --
	( 56.92, 78.92) --
	( 66.19, 77.10) --
	( 72.37, 80.57) --
	( 78.54, 75.95) --
	( 84.72, 77.15) --
	( 93.99, 75.20) --
	(100.17, 76.71) --
	(106.34, 78.57) --
	(115.61, 74.43) --
	(121.79, 77.18) --
	(127.97, 75.49) --
	(134.14, 74.89) --
	(143.41, 77.78) --
	(149.59, 78.37) --
	(155.77, 81.59) --
	(165.03, 78.88) --
	(171.21, 75.13) --
	(177.39, 82.21) --
	(183.56, 83.19) --
	(192.83, 78.19) --
	(199.01, 79.14) --
	(205.19, 83.61) --
	(211.36, 83.56) --
	(220.63, 81.85) --
	(226.81, 82.93) --
	(232.99, 84.64) --
	(242.25, 86.71) --
	(248.43, 87.09) --
	(254.61, 84.21) --
	(260.79, 89.02) --
	(270.05, 89.45) --
	(276.23, 86.38) --
	(282.41, 95.35) --
	(291.67, 96.33) --
	(297.85,111.12) --
	(304.03,111.53) --
	(310.21,119.40) --
	(319.47,142.10) --
	(325.65,186.06);
\definecolor{drawColor}{RGB}{238,44,44}

\path[draw=drawColor,line width= 1.4pt,line join=round] ( 50.74,236.27) --
	( 56.92,237.07) --
	( 66.19,236.24) --
	( 72.37,237.26) --
	( 78.54,235.79) --
	( 84.72,236.41) --
	( 93.99,234.89) --
	(100.17,236.18) --
	(106.34,236.28) --
	(115.61,234.66) --
	(121.79,235.19) --
	(127.97,234.73) --
	(134.14,235.39) --
	(143.41,234.89) --
	(149.59,234.89) --
	(155.77,235.65) --
	(165.03,235.11) --
	(171.21,233.81) --
	(177.39,234.76) --
	(183.56,234.55) --
	(192.83,234.32) --
	(199.01,234.77) --
	(205.19,235.79) --
	(211.36,234.62) --
	(220.63,234.36) --
	(226.81,234.03) --
	(232.99,234.15) --
	(242.25,232.74) --
	(248.43,232.55) --
	(254.61,232.22) --
	(260.79,230.89) --
	(270.05,229.96) --
	(276.23,228.63) --
	(282.41,228.67) --
	(291.67,227.83) --
	(297.85,224.31) --
	(304.03,222.26) --
	(310.21,219.14) --
	(319.47,205.65) --
	(325.65,186.05);
\end{scope}
\begin{scope}
\path[clip] (  0.00,  0.00) rectangle (433.62,289.08);
\definecolor{drawColor}{gray}{0.30}

\node[text=drawColor,anchor=base west,inner sep=0pt, outer sep=0pt, scale=  0.88] at ( 17.96, 38.41) {10};

\node[text=drawColor,anchor=base west,inner sep=0pt, outer sep=0pt, scale=  0.62] at ( 26.76, 42.00) {-};

\node[text=drawColor,anchor=base west,inner sep=0pt, outer sep=0pt, scale=  0.62] at ( 28.82, 42.00) {5};

\node[text=drawColor,anchor=base west,inner sep=0pt, outer sep=0pt, scale=  0.88] at ( 17.96, 95.88) {10};

\node[text=drawColor,anchor=base west,inner sep=0pt, outer sep=0pt, scale=  0.62] at ( 26.76, 99.48) {-};

\node[text=drawColor,anchor=base west,inner sep=0pt, outer sep=0pt, scale=  0.62] at ( 28.82, 99.48) {4};

\node[text=drawColor,anchor=base west,inner sep=0pt, outer sep=0pt, scale=  0.88] at ( 17.96,153.36) {10};

\node[text=drawColor,anchor=base west,inner sep=0pt, outer sep=0pt, scale=  0.62] at ( 26.76,156.96) {-};

\node[text=drawColor,anchor=base west,inner sep=0pt, outer sep=0pt, scale=  0.62] at ( 28.82,156.96) {3};

\node[text=drawColor,anchor=base west,inner sep=0pt, outer sep=0pt, scale=  0.88] at ( 17.96,210.83) {10};

\node[text=drawColor,anchor=base west,inner sep=0pt, outer sep=0pt, scale=  0.62] at ( 26.76,214.43) {-};

\node[text=drawColor,anchor=base west,inner sep=0pt, outer sep=0pt, scale=  0.62] at ( 28.82,214.43) {2};

\node[text=drawColor,anchor=base west,inner sep=0pt, outer sep=0pt, scale=  0.88] at ( 17.96,268.31) {10};

\node[text=drawColor,anchor=base west,inner sep=0pt, outer sep=0pt, scale=  0.62] at ( 26.76,271.91) {-};

\node[text=drawColor,anchor=base west,inner sep=0pt, outer sep=0pt, scale=  0.62] at ( 28.82,271.91) {1};
\end{scope}
\begin{scope}
\path[clip] (  0.00,  0.00) rectangle (433.62,289.08);
\definecolor{drawColor}{gray}{0.30}

\node[text=drawColor,anchor=base,inner sep=0pt, outer sep=0pt, scale=  0.88] at ( 50.74, 19.68) {10};

\node[text=drawColor,anchor=base,inner sep=0pt, outer sep=0pt, scale=  0.88] at ( 81.63, 19.68) {20};

\node[text=drawColor,anchor=base,inner sep=0pt, outer sep=0pt, scale=  0.88] at (112.52, 19.68) {30};

\node[text=drawColor,anchor=base,inner sep=0pt, outer sep=0pt, scale=  0.88] at (143.41, 19.68) {40};

\node[text=drawColor,anchor=base,inner sep=0pt, outer sep=0pt, scale=  0.88] at (174.30, 19.68) {50};

\node[text=drawColor,anchor=base,inner sep=0pt, outer sep=0pt, scale=  0.88] at (205.19, 19.68) {60};

\node[text=drawColor,anchor=base,inner sep=0pt, outer sep=0pt, scale=  0.88] at (236.08, 19.68) {70};

\node[text=drawColor,anchor=base,inner sep=0pt, outer sep=0pt, scale=  0.88] at (266.96, 19.68) {80};

\node[text=drawColor,anchor=base,inner sep=0pt, outer sep=0pt, scale=  0.88] at (297.85, 19.68) {90};

\node[text=drawColor,anchor=base,inner sep=0pt, outer sep=0pt, scale=  0.88] at (328.74, 19.68) {100};
\end{scope}
\begin{scope}
\path[clip] (  0.00,  0.00) rectangle (433.62,289.08);
\definecolor{drawColor}{RGB}{0,0,0}

\node[text=drawColor,anchor=base,inner sep=0pt, outer sep=0pt, scale=  1.10] at (189.74,  7.64) {$m$};
\end{scope}
\begin{scope}
\path[clip] (  0.00,  0.00) rectangle (433.62,289.08);
\definecolor{drawColor}{RGB}{0,0,0}

\node[text=drawColor,rotate= 90.00,anchor=base,inner sep=0pt, outer sep=0pt, scale=  1.10] at ( 13.08,157.13) {Value};
\end{scope}
\begin{scope}
\path[clip] (  0.00,  0.00) rectangle (433.62,289.08);
\definecolor{drawColor}{RGB}{70,130,180}

\path[draw=drawColor,line width= 1.4pt,line join=round] (360.59,156.75) -- (372.15,156.75);
\end{scope}
\begin{scope}
\path[clip] (  0.00,  0.00) rectangle (433.62,289.08);
\definecolor{drawColor}{RGB}{238,44,44}

\path[draw=drawColor,line width= 1.4pt,line join=round] (360.59,142.30) -- (372.15,142.30);
\end{scope}
\begin{scope}
\path[clip] (  0.00,  0.00) rectangle (433.62,289.08);
\definecolor{drawColor}{RGB}{0,0,0}

\node[text=drawColor,anchor=base west,inner sep=0pt, outer sep=0pt, scale=  0.88] at (379.09,153.72) {$\overline{F}$};
\end{scope}
\begin{scope}
\path[clip] (  0.00,  0.00) rectangle (433.62,289.08);
\definecolor{drawColor}{RGB}{0,0,0}

\node[text=drawColor,anchor=base west,inner sep=0pt, outer sep=0pt, scale=  0.88] at (379.09,139.27) {$\overline{T}$};
\end{scope}
\end{tikzpicture}

%% file: figures/delta_sweep_plot_v3.tex
\begin{tikzpicture}[x=1pt,y=1pt]
\definecolor{fillColor}{RGB}{255,255,255}
\path[use as bounding box,fill=fillColor,fill opacity=0.00] (0,0) rectangle (433.62,289.08);
\begin{scope}
\path[clip] (  0.00,  0.00) rectangle (433.62,289.08);
\definecolor{fillColor}{RGB}{255,255,255}

\path[fill=fillColor] (  0.00,  0.00) rectangle (433.62,289.08);
\end{scope}
\begin{scope}
\path[clip] ( 36.84, 30.69) rectangle (342.64,283.58);
\definecolor{drawColor}{gray}{0.92}

\path[draw=drawColor,line width= 0.3pt,line join=round] ( 36.84, 65.17) --
	(342.64, 65.17);

\path[draw=drawColor,line width= 0.3pt,line join=round] ( 36.84,111.15) --
	(342.64,111.15);

\path[draw=drawColor,line width= 0.3pt,line join=round] ( 36.84,157.13) --
	(342.64,157.13);

\path[draw=drawColor,line width= 0.3pt,line join=round] ( 36.84,203.11) --
	(342.64,203.11);

\path[draw=drawColor,line width= 0.3pt,line join=round] ( 36.84,249.09) --
	(342.64,249.09);

\path[draw=drawColor,line width= 0.3pt,line join=round] ( 52.13, 30.69) --
	( 52.13,283.58);

\path[draw=drawColor,line width= 0.3pt,line join=round] ( 82.71, 30.69) --
	( 82.71,283.58);

\path[draw=drawColor,line width= 0.3pt,line join=round] (113.29, 30.69) --
	(113.29,283.58);

\path[draw=drawColor,line width= 0.3pt,line join=round] (143.87, 30.69) --
	(143.87,283.58);

\path[draw=drawColor,line width= 0.3pt,line join=round] (174.45, 30.69) --
	(174.45,283.58);

\path[draw=drawColor,line width= 0.3pt,line join=round] (205.03, 30.69) --
	(205.03,283.58);

\path[draw=drawColor,line width= 0.3pt,line join=round] (235.61, 30.69) --
	(235.61,283.58);

\path[draw=drawColor,line width= 0.3pt,line join=round] (266.19, 30.69) --
	(266.19,283.58);

\path[draw=drawColor,line width= 0.3pt,line join=round] (296.77, 30.69) --
	(296.77,283.58);

\path[draw=drawColor,line width= 0.3pt,line join=round] (327.35, 30.69) --
	(327.35,283.58);

\path[draw=drawColor,line width= 0.6pt,line join=round] ( 36.84, 42.18) --
	(342.64, 42.18);

\path[draw=drawColor,line width= 0.6pt,line join=round] ( 36.84, 88.16) --
	(342.64, 88.16);

\path[draw=drawColor,line width= 0.6pt,line join=round] ( 36.84,134.14) --
	(342.64,134.14);

\path[draw=drawColor,line width= 0.6pt,line join=round] ( 36.84,180.12) --
	(342.64,180.12);

\path[draw=drawColor,line width= 0.6pt,line join=round] ( 36.84,226.10) --
	(342.64,226.10);

\path[draw=drawColor,line width= 0.6pt,line join=round] ( 36.84,272.08) --
	(342.64,272.08);

\path[draw=drawColor,line width= 0.6pt,line join=round] ( 36.84, 30.69) --
	( 36.84,283.58);

\path[draw=drawColor,line width= 0.6pt,line join=round] ( 67.42, 30.69) --
	( 67.42,283.58);

\path[draw=drawColor,line width= 0.6pt,line join=round] ( 98.00, 30.69) --
	( 98.00,283.58);

\path[draw=drawColor,line width= 0.6pt,line join=round] (128.58, 30.69) --
	(128.58,283.58);

\path[draw=drawColor,line width= 0.6pt,line join=round] (159.16, 30.69) --
	(159.16,283.58);

\path[draw=drawColor,line width= 0.6pt,line join=round] (189.74, 30.69) --
	(189.74,283.58);

\path[draw=drawColor,line width= 0.6pt,line join=round] (220.32, 30.69) --
	(220.32,283.58);

\path[draw=drawColor,line width= 0.6pt,line join=round] (250.90, 30.69) --
	(250.90,283.58);

\path[draw=drawColor,line width= 0.6pt,line join=round] (281.48, 30.69) --
	(281.48,283.58);

\path[draw=drawColor,line width= 0.6pt,line join=round] (312.06, 30.69) --
	(312.06,283.58);

\path[draw=drawColor,line width= 0.6pt,line join=round] (342.64, 30.69) --
	(342.64,283.58);
\definecolor{fillColor}{RGB}{70,130,180}

\path[fill=fillColor,fill opacity=0.20] ( 61.13, 81.56) --
	( 68.35,102.69) --
	( 75.57,115.54) --
	( 82.79,124.31) --
	( 90.01,134.38) --
	( 97.23,137.86) --
	(104.44,145.20) --
	(111.66,150.29) --
	(118.88,157.85) --
	(126.10,159.65) --
	(133.32,164.28) --
	(140.53,168.04) --
	(147.75,173.83) --
	(154.97,179.78) --
	(162.19,182.85) --
	(169.41,184.93) --
	(176.62,186.73) --
	(183.84,188.59) --
	(191.06,194.35) --
	(198.28,194.84) --
	(205.50,199.61) --
	(212.71,201.32) --
	(219.93,203.72) --
	(227.15,206.92) --
	(234.37,209.24) --
	(241.59,210.05) --
	(248.81,215.17) --
	(256.02,216.32) --
	(263.24,218.87) --
	(270.46,222.42) --
	(277.68,223.17) --
	(284.90,225.48) --
	(292.11,226.59) --
	(299.33,228.26) --
	(306.55,232.04) --
	(313.77,234.29) --
	(320.99,237.50) --
	(328.20,238.00) --
	(335.42,243.24) --
	(342.64,272.08) --
	(342.64,272.08) --
	(335.42,243.23) --
	(328.20,237.99) --
	(320.99,237.49) --
	(313.77,234.29) --
	(306.55,232.03) --
	(299.33,228.26) --
	(292.11,226.58) --
	(284.90,225.47) --
	(277.68,223.16) --
	(270.46,222.41) --
	(263.24,218.86) --
	(256.02,216.31) --
	(248.81,215.16) --
	(241.59,210.04) --
	(234.37,209.22) --
	(227.15,206.91) --
	(219.93,203.71) --
	(212.71,201.30) --
	(205.50,199.59) --
	(198.28,194.82) --
	(191.06,194.33) --
	(183.84,188.57) --
	(176.62,186.71) --
	(169.41,184.91) --
	(162.19,182.83) --
	(154.97,179.75) --
	(147.75,173.80) --
	(140.53,168.01) --
	(133.32,164.24) --
	(126.10,159.60) --
	(118.88,157.81) --
	(111.66,150.24) --
	(104.44,145.14) --
	( 97.23,137.79) --
	( 90.01,134.30) --
	( 82.79,124.21) --
	( 75.57,115.42) --
	( 68.35,102.52) --
	( 61.13, 81.26) --
	cycle;

\path[] ( 61.13, 81.56) --
	( 68.35,102.69) --
	( 75.57,115.54) --
	( 82.79,124.31) --
	( 90.01,134.38) --
	( 97.23,137.86) --
	(104.44,145.20) --
	(111.66,150.29) --
	(118.88,157.85) --
	(126.10,159.65) --
	(133.32,164.28) --
	(140.53,168.04) --
	(147.75,173.83) --
	(154.97,179.78) --
	(162.19,182.85) --
	(169.41,184.93) --
	(176.62,186.73) --
	(183.84,188.59) --
	(191.06,194.35) --
	(198.28,194.84) --
	(205.50,199.61) --
	(212.71,201.32) --
	(219.93,203.72) --
	(227.15,206.92) --
	(234.37,209.24) --
	(241.59,210.05) --
	(248.81,215.17) --
	(256.02,216.32) --
	(263.24,218.87) --
	(270.46,222.42) --
	(277.68,223.17) --
	(284.90,225.48) --
	(292.11,226.59) --
	(299.33,228.26) --
	(306.55,232.04) --
	(313.77,234.29) --
	(320.99,237.50) --
	(328.20,238.00) --
	(335.42,243.24) --
	(342.64,272.08);

\path[] (342.64,272.08) --
	(335.42,243.23) --
	(328.20,237.99) --
	(320.99,237.49) --
	(313.77,234.29) --
	(306.55,232.03) --
	(299.33,228.26) --
	(292.11,226.58) --
	(284.90,225.47) --
	(277.68,223.16) --
	(270.46,222.41) --
	(263.24,218.86) --
	(256.02,216.31) --
	(248.81,215.16) --
	(241.59,210.04) --
	(234.37,209.22) --
	(227.15,206.91) --
	(219.93,203.71) --
	(212.71,201.30) --
	(205.50,199.59) --
	(198.28,194.82) --
	(191.06,194.33) --
	(183.84,188.57) --
	(176.62,186.71) --
	(169.41,184.91) --
	(162.19,182.83) --
	(154.97,179.75) --
	(147.75,173.80) --
	(140.53,168.01) --
	(133.32,164.24) --
	(126.10,159.60) --
	(118.88,157.81) --
	(111.66,150.24) --
	(104.44,145.14) --
	( 97.23,137.79) --
	( 90.01,134.30) --
	( 82.79,124.21) --
	( 75.57,115.42) --
	( 68.35,102.52) --
	( 61.13, 81.26);
\definecolor{fillColor}{RGB}{238,44,44}

\path[fill=fillColor,fill opacity=0.15] ( 61.13,187.41) --
	( 68.35,196.63) --
	( 75.57,202.05) --
	( 82.79,207.81) --
	( 90.01,211.70) --
	( 97.23,215.85) --
	(104.44,219.06) --
	(111.66,222.06) --
	(118.88,224.96) --
	(126.10,227.17) --
	(133.32,229.60) --
	(140.53,232.14) --
	(147.75,234.83) --
	(154.97,237.21) --
	(162.19,239.37) --
	(169.41,240.47) --
	(176.62,242.56) --
	(183.84,244.08) --
	(191.06,245.82) --
	(198.28,246.84) --
	(205.50,248.95) --
	(212.71,250.04) --
	(219.93,251.14) --
	(227.15,252.89) --
	(234.37,253.84) --
	(241.59,255.32) --
	(248.81,256.53) --
	(256.02,257.66) --
	(263.24,259.14) --
	(270.46,260.70) --
	(277.68,261.52) --
	(284.90,263.02) --
	(292.11,263.85) --
	(299.33,265.23) --
	(306.55,266.60) --
	(313.77,267.92) --
	(320.99,269.22) --
	(328.20,270.58) --
	(335.42,271.58) --
	(342.64,272.08) --
	(342.64,272.08) --
	(335.42,271.51) --
	(328.20,270.46) --
	(320.99,269.07) --
	(313.77,267.75) --
	(306.55,266.40) --
	(299.33,265.02) --
	(292.11,263.62) --
	(284.90,262.79) --
	(277.68,261.25) --
	(270.46,260.43) --
	(263.24,258.83) --
	(256.02,257.31) --
	(248.81,256.16) --
	(241.59,254.94) --
	(234.37,253.40) --
	(227.15,252.46) --
	(219.93,250.66) --
	(212.71,249.53) --
	(205.50,248.43) --
	(198.28,246.27) --
	(191.06,245.24) --
	(183.84,243.46) --
	(176.62,241.92) --
	(169.41,239.77) --
	(162.19,238.66) --
	(154.97,236.46) --
	(147.75,234.02) --
	(140.53,231.30) --
	(133.32,228.72) --
	(126.10,226.21) --
	(118.88,224.00) --
	(111.66,221.06) --
	(104.44,217.99) --
	( 97.23,214.73) --
	( 90.01,210.51) --
	( 82.79,206.57) --
	( 75.57,200.67) --
	( 68.35,195.18) --
	( 61.13,185.87) --
	cycle;

\path[] ( 61.13,187.41) --
	( 68.35,196.63) --
	( 75.57,202.05) --
	( 82.79,207.81) --
	( 90.01,211.70) --
	( 97.23,215.85) --
	(104.44,219.06) --
	(111.66,222.06) --
	(118.88,224.96) --
	(126.10,227.17) --
	(133.32,229.60) --
	(140.53,232.14) --
	(147.75,234.83) --
	(154.97,237.21) --
	(162.19,239.37) --
	(169.41,240.47) --
	(176.62,242.56) --
	(183.84,244.08) --
	(191.06,245.82) --
	(198.28,246.84) --
	(205.50,248.95) --
	(212.71,250.04) --
	(219.93,251.14) --
	(227.15,252.89) --
	(234.37,253.84) --
	(241.59,255.32) --
	(248.81,256.53) --
	(256.02,257.66) --
	(263.24,259.14) --
	(270.46,260.70) --
	(277.68,261.52) --
	(284.90,263.02) --
	(292.11,263.85) --
	(299.33,265.23) --
	(306.55,266.60) --
	(313.77,267.92) --
	(320.99,269.22) --
	(328.20,270.58) --
	(335.42,271.58) --
	(342.64,272.08);

\path[] (342.64,272.08) --
	(335.42,271.51) --
	(328.20,270.46) --
	(320.99,269.07) --
	(313.77,267.75) --
	(306.55,266.40) --
	(299.33,265.02) --
	(292.11,263.62) --
	(284.90,262.79) --
	(277.68,261.25) --
	(270.46,260.43) --
	(263.24,258.83) --
	(256.02,257.31) --
	(248.81,256.16) --
	(241.59,254.94) --
	(234.37,253.40) --
	(227.15,252.46) --
	(219.93,250.66) --
	(212.71,249.53) --
	(205.50,248.43) --
	(198.28,246.27) --
	(191.06,245.24) --
	(183.84,243.46) --
	(176.62,241.92) --
	(169.41,239.77) --
	(162.19,238.66) --
	(154.97,236.46) --
	(147.75,234.02) --
	(140.53,231.30) --
	(133.32,228.72) --
	(126.10,226.21) --
	(118.88,224.00) --
	(111.66,221.06) --
	(104.44,217.99) --
	( 97.23,214.73) --
	( 90.01,210.51) --
	( 82.79,206.57) --
	( 75.57,200.67) --
	( 68.35,195.18) --
	( 61.13,185.87);
\definecolor{drawColor}{RGB}{70,130,180}

\path[draw=drawColor,line width= 1.4pt,line join=round] ( 61.13, 81.41) --
	( 68.35,102.60) --
	( 75.57,115.48) --
	( 82.79,124.26) --
	( 90.01,134.34) --
	( 97.23,137.83) --
	(104.44,145.17) --
	(111.66,150.26) --
	(118.88,157.83) --
	(126.10,159.63) --
	(133.32,164.26) --
	(140.53,168.03) --
	(147.75,173.82) --
	(154.97,179.76) --
	(162.19,182.84) --
	(169.41,184.92) --
	(176.62,186.72) --
	(183.84,188.58) --
	(191.06,194.34) --
	(198.28,194.83) --
	(205.50,199.60) --
	(212.71,201.31) --
	(219.93,203.71) --
	(227.15,206.91) --
	(234.37,209.23) --
	(241.59,210.05) --
	(248.81,215.17) --
	(256.02,216.31) --
	(263.24,218.87) --
	(270.46,222.41) --
	(277.68,223.16) --
	(284.90,225.47) --
	(292.11,226.58) --
	(299.33,228.26) --
	(306.55,232.03) --
	(313.77,234.29) --
	(320.99,237.49) --
	(328.20,238.00) --
	(335.42,243.24) --
	(342.64,272.08);
\definecolor{drawColor}{RGB}{238,44,44}

\path[draw=drawColor,line width= 1.4pt,line join=round] ( 61.13,186.66) --
	( 68.35,195.92) --
	( 75.57,201.38) --
	( 82.79,207.20) --
	( 90.01,211.11) --
	( 97.23,215.30) --
	(104.44,218.53) --
	(111.66,221.57) --
	(118.88,224.48) --
	(126.10,226.69) --
	(133.32,229.16) --
	(140.53,231.72) --
	(147.75,234.43) --
	(154.97,236.84) --
	(162.19,239.02) --
	(169.41,240.12) --
	(176.62,242.24) --
	(183.84,243.77) --
	(191.06,245.53) --
	(198.28,246.55) --
	(205.50,248.69) --
	(212.71,249.79) --
	(219.93,250.90) --
	(227.15,252.67) --
	(234.37,253.62) --
	(241.59,255.13) --
	(248.81,256.35) --
	(256.02,257.49) --
	(263.24,258.98) --
	(270.46,260.57) --
	(277.68,261.39) --
	(284.90,262.90) --
	(292.11,263.73) --
	(299.33,265.12) --
	(306.55,266.50) --
	(313.77,267.84) --
	(320.99,269.14) --
	(328.20,270.52) --
	(335.42,271.54) --
	(342.64,272.08);
\end{scope}
\begin{scope}
\path[clip] (  0.00,  0.00) rectangle (433.62,289.08);
\definecolor{drawColor}{gray}{0.30}

\node[text=drawColor,anchor=base west,inner sep=0pt, outer sep=0pt, scale=  0.88] at ( 17.96, 38.41) {10};

\node[text=drawColor,anchor=base west,inner sep=0pt, outer sep=0pt, scale=  0.62] at ( 26.76, 42.00) {-};

\node[text=drawColor,anchor=base west,inner sep=0pt, outer sep=0pt, scale=  0.62] at ( 28.82, 42.00) {5};

\node[text=drawColor,anchor=base west,inner sep=0pt, outer sep=0pt, scale=  0.88] at ( 17.96, 84.39) {10};

\node[text=drawColor,anchor=base west,inner sep=0pt, outer sep=0pt, scale=  0.62] at ( 26.76, 87.99) {-};

\node[text=drawColor,anchor=base west,inner sep=0pt, outer sep=0pt, scale=  0.62] at ( 28.82, 87.99) {4};

\node[text=drawColor,anchor=base west,inner sep=0pt, outer sep=0pt, scale=  0.88] at ( 17.96,130.37) {10};

\node[text=drawColor,anchor=base west,inner sep=0pt, outer sep=0pt, scale=  0.62] at ( 26.76,133.97) {-};

\node[text=drawColor,anchor=base west,inner sep=0pt, outer sep=0pt, scale=  0.62] at ( 28.82,133.97) {3};

\node[text=drawColor,anchor=base west,inner sep=0pt, outer sep=0pt, scale=  0.88] at ( 17.96,176.35) {10};

\node[text=drawColor,anchor=base west,inner sep=0pt, outer sep=0pt, scale=  0.62] at ( 26.76,179.95) {-};

\node[text=drawColor,anchor=base west,inner sep=0pt, outer sep=0pt, scale=  0.62] at ( 28.82,179.95) {2};

\node[text=drawColor,anchor=base west,inner sep=0pt, outer sep=0pt, scale=  0.88] at ( 17.96,222.33) {10};

\node[text=drawColor,anchor=base west,inner sep=0pt, outer sep=0pt, scale=  0.62] at ( 26.76,225.93) {-};

\node[text=drawColor,anchor=base west,inner sep=0pt, outer sep=0pt, scale=  0.62] at ( 28.82,225.93) {1};

\node[text=drawColor,anchor=base west,inner sep=0pt, outer sep=0pt, scale=  0.88] at ( 20.02,268.31) {10};

\node[text=drawColor,anchor=base west,inner sep=0pt, outer sep=0pt, scale=  0.62] at ( 28.82,271.91) {0};
\end{scope}
\begin{scope}
\path[clip] (  0.00,  0.00) rectangle (433.62,289.08);
\definecolor{drawColor}{gray}{0.30}

\node[text=drawColor,anchor=base,inner sep=0pt, outer sep=0pt, scale=  0.88] at ( 36.84, 19.68) {0.0};

\node[text=drawColor,anchor=base,inner sep=0pt, outer sep=0pt, scale=  0.88] at ( 67.42, 19.68) {0.1};

\node[text=drawColor,anchor=base,inner sep=0pt, outer sep=0pt, scale=  0.88] at ( 98.00, 19.68) {0.2};

\node[text=drawColor,anchor=base,inner sep=0pt, outer sep=0pt, scale=  0.88] at (128.58, 19.68) {0.3};

\node[text=drawColor,anchor=base,inner sep=0pt, outer sep=0pt, scale=  0.88] at (159.16, 19.68) {0.4};

\node[text=drawColor,anchor=base,inner sep=0pt, outer sep=0pt, scale=  0.88] at (189.74, 19.68) {0.5};

\node[text=drawColor,anchor=base,inner sep=0pt, outer sep=0pt, scale=  0.88] at (220.32, 19.68) {0.6};

\node[text=drawColor,anchor=base,inner sep=0pt, outer sep=0pt, scale=  0.88] at (250.90, 19.68) {0.7};

\node[text=drawColor,anchor=base,inner sep=0pt, outer sep=0pt, scale=  0.88] at (281.48, 19.68) {0.8};

\node[text=drawColor,anchor=base,inner sep=0pt, outer sep=0pt, scale=  0.88] at (312.06, 19.68) {0.9};

\node[text=drawColor,anchor=base,inner sep=0pt, outer sep=0pt, scale=  0.88] at (342.64, 19.68) {1.0};
\end{scope}
\begin{scope}
\path[clip] (  0.00,  0.00) rectangle (433.62,289.08);
\definecolor{drawColor}{RGB}{0,0,0}

\node[text=drawColor,anchor=base,inner sep=0pt, outer sep=0pt, scale=  1.10] at (189.74,  7.64) {$\delta$};
\end{scope}
\begin{scope}
\path[clip] (  0.00,  0.00) rectangle (433.62,289.08);
\definecolor{drawColor}{RGB}{0,0,0}

\node[text=drawColor,rotate= 90.00,anchor=base,inner sep=0pt, outer sep=0pt, scale=  1.10] at ( 13.08,157.13) {Value};
\end{scope}
\begin{scope}
\path[clip] (  0.00,  0.00) rectangle (433.62,289.08);
\definecolor{drawColor}{RGB}{70,130,180}

\path[draw=drawColor,line width= 1.4pt,line join=round] (360.59,156.75) -- (372.15,156.75);
\end{scope}
\begin{scope}
\path[clip] (  0.00,  0.00) rectangle (433.62,289.08);
\definecolor{drawColor}{RGB}{238,44,44}

\path[draw=drawColor,line width= 1.4pt,line join=round] (360.59,142.30) -- (372.15,142.30);
\end{scope}
\begin{scope}
\path[clip] (  0.00,  0.00) rectangle (433.62,289.08);
\definecolor{drawColor}{RGB}{0,0,0}

\node[text=drawColor,anchor=base west,inner sep=0pt, outer sep=0pt, scale=  0.88] at (379.09,153.72) {$\overline{F}$};
\end{scope}
\begin{scope}
\path[clip] (  0.00,  0.00) rectangle (433.62,289.08);
\definecolor{drawColor}{RGB}{0,0,0}

\node[text=drawColor,anchor=base west,inner sep=0pt, outer sep=0pt, scale=  0.88] at (379.09,139.27) {$\overline{T}$};
\end{scope}
\end{tikzpicture}

%% file: conclusion.tex
In this paper, we considered data universes, data-release mechanisms, approximate differential privacy, and probabilistic differential privacy within a common measure-theoretic framework. 
With this unified formulation, we rederived the well-known result that probabilistic differential privacy implies approximate differential privacy, as well as the converse implication with modified $(\varepsilon,\delta)$-parameters. 

Using the implication from approximate differential privacy to probabilistic differential privacy, we derived upper and lower bounds for posterior-to-prior ratios of inclusion beliefs for data-release mechanisms satisfying bounded probabilistic or approximate differential privacy.
These bounds quantify the extent to which observing a mechanism output can alter prior beliefs about inclusion events. Since the bounds are only guaranteed when the mechanism output falls outside certain regions of the output space, we further studied the probability that an output lies within such regions. Using a geometric approach, we derived an exact expression for this probability for the Gaussian mechanism. However, as violations of the bounds are not guaranteed within these regions, this probability constitutes an upper limit on the failure probability. To investigate the potential gap between the failure probability and its theoretical upper limit, we conducted a Monte Carlo study. The simulation results indicate that the theoretical upper limit is conservative, lying substantially above the estimated failure probability throughout the parameter sweeps considered in the study.

These findings aid in the assessment of privacy risks associated with the release of differentially private data. The posterior-to-prior ratio bounds can help a data holder select privacy parameters that are consistent with a predetermined acceptable level of privacy risk. For the Gaussian mechanism, the theoretical upper limit on the failure probability can be calculated using the formula derived in this paper. Since this upper limit appears to be considerably larger than the actual failure probability, calibrating the mechanism to keep the limit below an acceptable level provides a conservative privacy guarantee.

Although our framework provides rigorous results, several limitations remain. First, our derivations of the posterior-to-prior ratio bounds and the theoretical upper limit on the failure probability rely on the assumption that the underlying population is finite and known. Extending these results to settings with unknown, infinite, or uncountable populations would therefore be a valuable direction for future research. Second, while we only consider the posterior-to-prior ratios for inclusion events, adverse inferences encompass a broader range of inference types, including attribute inference and data reconstruction attacks. Extending the bounds to a wider class of inference problems would be an important direction for future research. Finally, the true failure probability of the posterior-to-prior ratio bounds may depend strongly on the specific data-release mechanism. While the theoretical upper limit was analytically tractable for the Gaussian mechanism, analogous derivations may be intractable for more general mechanisms. Investigating practical methods for estimating failure probabilities under arbitrary data-release mechanisms is therefore a highly relevant avenue for future work.